\documentclass[11pt, leqno]{amsart}
\usepackage[T1]{fontenc}
\usepackage[foot]{amsaddr}
\usepackage{amsmath}%
\usepackage{amsthm}
\usepackage{bbm}
\usepackage{enumerate}
\usepackage{amssymb}%
\usepackage[margin=1.25in]{geometry}
\usepackage{color,hyperref}
\usepackage{graphicx}
\usepackage{subfig}
\usepackage{caption}
\usepackage{cite}
\usepackage{dsfont}
\hypersetup{colorlinks,breaklinks,
             linkcolor=blue,urlcolor=blue,
             anchorcolor=blue,citecolor=blue}

\title{Rates of convergence for the continuum limit of nondominated sorting}
\author{Brendan Cook \and Jeff Calder}
\address[B.~Cook, J.~Calder]{School of Mathematics, University of Minnesota}
\email{cookx932@umn.edu}
\email{jcalder@umn.edu}
\thanks{B.~Cook and J.~Calder were supported by NSF-DMS grant 1713691 and a University of Minnesota Grant in Aid award.}

\newcommand{\Z}{\mathbb{Z}}
\newcommand{\R}{\mathbb{R}}

\newcommand{\C}{\mathcal{C}}
\newcommand{\N}{\mathbb{N}}

\renewcommand{\O}{{\mathcal O}}

\newcommand{\E}{\mathbb{E}}

\newcommand{\s}{\mathcal{S}}

\newcommand{\p}{\mathbb{P}}
\newcommand{\vare}{\varepsilon}

\newcommand{\ov}[1]{\overline{#1}}
\newcommand{\uv}[1]{\underline{#1}}

\newcommand{\set}[1]{\left\{#1 \right\}}
\newcommand{\dif}[1]{\text{ }d #1}

\newcommand{\abs}[1]{\left\lvert#1\right\rvert}
\newcommand{\norm}[1]{\left\lVert#1\right\rVert}
\newcommand{\rhomin}{\rho_{\textrm{min}}}
\newcommand{\rhomax}{\rho_{\textrm{max}}}
\newcommand{\Czeroone}{C^{0,1}(\R^d)}
\newcommand{\Ctwo}{C^2(\R^d)}

\DeclareMathOperator{\dist}{dist}

\DeclareMathOperator{\argmin}{argmin}

\DeclareMathOperator{\USC}{USC}
\DeclareMathOperator{\LSC}{LSC}

\newtheorem{lemma}{Lemma}[section]
\newtheorem{proposition}{Proposition}[section]
\newtheorem{theorem}{Theorem}[section]
\newtheorem{remark}{Remark}[section]
\newtheorem{corollary}{Corollary}[section]
\newtheorem{definition}{Definition}[section]
\numberwithin{equation}{section}
\numberwithin{theorem}{section}
\numberwithin{proposition}{section}

\def\XXint#1#2#3{{\setbox0=\hbox{$#1{#2#3}{\int}$ }
\vcenter{\hbox{$#2#3$ }}\kern-.6\wd0}}

\begin{document}
\begin{abstract}
    Nondominated sorting is a discrete process that sorts points in Euclidean space according to the coordinatewise partial order, and is used to rank feasible solutions to multiobjective optimization problems. It was previously shown that nondominated sorting of random points has a Hamilton-Jacobi equation continuum limit. We prove quantitative error estimates for the convergence of nondominated sorting to its continuum limit Hamilton-Jacobi equation.  Our proof uses the maximum principle and viscosity solution machinery, along with new semiconvexity estimates for domains with corner singularities.
\end{abstract}
\maketitle

\section{Introduction}\label{sec:introduction}
The sorting of multivariate data is an important problem in many fields of applied science \cite{calder2014}. \textit{Nondominated sorting} is a discrete process that is widely applied in multiobjective optimization and can be interpreted as arranging a finite set of points in Euclidean space into layers according to the coordinatewise partial order. 
Let $\leq$ denote the coordinatewise partial order on $\R^d$ given by
\begin{align*}
    x\leq y \iff x_i \leq y_i \text{ for all } i=1,\ldots , d.
\end{align*}
Given a set of distinct points $X = \set{X_1\ldots ,X_n} \subset \R^d$, let $\mathcal{F}_1$ denote the subset of points that are coordinatewise minimal. The set $\mathcal{F}_1$ is called the \textit{first Pareto front}, and the elements of $\mathcal{F}_1$ are called \textit{Pareto-optimal} or \textit{nondominated}. In general, the $k$-th Pareto front is defined by
\begin{align*}
    \mathcal{F}_k = \text{Minimal elements of } X \setminus \bigcup_{j<k} \mathcal{F}_j ,
\end{align*}
and \textit{nondominated sorting} is the process of sorting a given set of points by Pareto-optimality. A multiobjective optimization problem involves identifying from a given set of feasible solutions those that minimize a collection of objective functions. In the context of multiobjective optimization, the $d$ coordinates of a point to be sorted are the values of the $d$ objective functions on a given feasible solution, and nondominated sorting provides an effective ranking of all feasible solutions. Nondominated sorting and multiobjective optimization are widely used in science and engineering disciplines \cite{deb2002, ehrgott2005}, particularly to control theory and path planning \cite{kumar2010, mitchell2003}, gene selection \cite{fleury2002, hero2003}, clustering \cite{handl2005}, anomaly detection \cite{hsiao2015b, hsiao2012}, and image processing \cite{mumford1989, chan2001, hsiao2015}.

Set $\R^d_+ = \set{x\in \R^d: x_i > 0 \text{ for } i=1,\ldots ,d}$ and define the \textit{Pareto-depth function} $U_n = \sum_{j=1}^n \mathbbm{1}_{P_j}$ where $P_j = \set{x\in \R^d_+ : x\geq y \text{ for some } y\in \mathcal{F}_j}$. It was shown in \cite{calder2014} that if the $X_i$ are \textit{i.i.d.}~random variables on $\R^d_+$ with density $\rho$, then $n^{-1/d} U_n \to \frac{c_d}{d} u$ almost surely in $L^\infty(\R^d)$ as $n\to \infty$ where $u$ is the unique nondecreasing viscosity solution of the problem
\begin{equation}\label{eq:PDE}
\left\{\begin{aligned}
u_{x_1}u_{x_2}\ldots u_{x_d}  &=  \rho&&\text{in }\R^d_+\\ 
u &= 0 &&\text{on }\partial \R^d_+.
\end{aligned}\right.
\end{equation}
and $c_d$ is a constant. This result shows that nondominated sorting of large datasets can be approximated by solving a partial differential equation numerically. This idea was developed further by Calder et al. in \cite{calder2015PDE} which proposed a fast approximate algorithm for nondominated sorting called \textit{PDE-based ranking} based on estimating $\rho$ from the data and solving the PDE numerically. It was shown in \cite{calder2015PDE} that PDE-based ranking is considerably faster than nondominated sorting in low dimensions while maintaining high sorting accuracy.

In this paper, we establish rates of convergence for the continuum limit of nondominated sorting. This is an important result in applications of PDE-based ranking \cite{abbasi2017, hsiao2015b} where it is important to consider how the error scales with the size $n$ of the dataset. The problem has several features that complicate the proof. The Hamiltonian $H(p) = p_1 \ldots p_d$ is not coercive, which is the standard property required to prove Lipschitz regularity of viscosity solutions \cite{bardi1997}. If one takes a $d$th root of the PDE to replace the Hamiltonian with $H(p) = (p_1 \ldots p_d)^{1/d}$, we obtain a concave $H$ at the cost of losing local Lipschitz regularity. In particular, solutions of \eqref{eq:PDE} are neither semiconcave nor semiconvex in general. Furthermore, $u$ is not Lipschitz due to the lack of boundary smoothness and coercivity. Our proof approximates the solution to \eqref{eq:PDE} by the solution to the auxiliary problem
\begin{equation}\label{eq:PDEaux}
\left\{\begin{aligned}
u_{x_1}u_{x_2}\ldots u_{x_d}  &=  \rho&&\text{in } \Omega_R\\ 
u &= 0 &&\text{on }\partial_R \Omega,
\end{aligned}\right.
\end{equation}
where $\Omega_R = \set{x\in [0,1]^d: (x_1\ldots x_d)^{1/d} > R}$ and $\partial_R \Omega = \set{x\in [0,1]^d : (x_1\ldots x_d)^{1/d} = R}$, effectively rounding off the corner singularity. We prove a one-sided convergence rate for the auxiliary problem restricted to the box $[0,1]^d$ by using an $\inf$-convolution to approximate $u$ by semiconcave functions that solve \eqref{eq:PDEaux} approximately. We apply the convergence rates for the longest chain problem proved in \cite{bollobas1992height} to obtain rates that hold with high probability on a collection of simplices, which are essentially cell-problems from homogenization theory. The remainder of the argument builds off of the proof in \cite{calder2016direct} but keeping track quantitatively of all sources of error. 

We also prove new semiconvexity results on the corner domain $\R^d_+$, which bound the blowup rate of the semiconvexity constant of $u$ at the boundary. The semiconvex regularity of $u$ on the auxiliary domain enables us to avoid use of a sup-convolution approximation for this direction, bolstering the convergence rate. The proof uses a closed-form asymptotic expansion to obtain a smooth approximate solution to \eqref{eq:PDEaux} near the boundary, and computes semiconvexity estimates for the approximation analytically. We believe this argument is new, as the typical arguments found in the literature for proving semiconvexity near the boundary proceed by means of vanishing viscosity \cite{bardi1997}. We also extend the semiconvexity estimates to the full domain with a doubling variables argument which is new and simpler compared to the standard tripling variables approach \cite{bardi1997}.

Our convergence rate proof is at a high level similar to the proofs of convergence rates for stochastic homogenization of Hamilton-Jacobi equations in \cite{armstrong2014error}, which uses Azuma's inequality to control fluctuations and a doubling variables argument to prove convergence rates. Apart from the viscosity solution theory, the main machinery we use is the convergence rate for the longest chain problem proved by Bollob\'as and Brightwell in \cite{bollobas1992height}, whose proof is also based on Azuma's inequality. As our PDE is first-order, our approach uses the $\inf$ convolution instead of a doubling variables argument which leads to an equivalent but somewhat simplified argument.

As described in \cite{calder2016direct}, this continuum limit result can be viewed in the context of stochastic homogenization of Hamilton-Jacobi equations. One may interpret $U_n$ as the discontinuous viscosity solution of 
\begin{equation}\label{eq:UnPDE}
\left\{\begin{aligned}
U_{n,x_1}U_{n,x_2}\ldots U_{n,x_d}  &=  \sum_{j=1}^n \delta_{X_j} &&\text{in }\R^d_+\\ 
U_n &= 0 &&\text{on }\partial \R^d_+.
\end{aligned}\right.
\end{equation}
The sense in which $U_n$ solves the PDE (\ref{eq:UnPDE}) is not obvious. By mollifying $U_n$, one obtains a sequence $U_n^\vare$ of approximate solutions to (\ref{eq:UnPDE}). It can be shown that  $U_n^\vare$ converges pointwise to $CU_n$ as $\vare \to 0$ where the constant $C$ depends on the choice of mollification kernel. 

Our proof techniques may also be applicable to several other related problems in the literature. The convex peeling problem studied in \cite{calder2020convex} bears many similarities to our problem, and similar ideas may give convergence rates for the convex peeling problem, provided the solutions of the continuum PDE are sufficiently smooth. The papers \cite{thawinrak2017high, calder2017numerical} introduce numerical methods for the PDE (\ref{eq:PDE}) and prove convergence rates. Our semiconvex regularity results could be used to improve the convergence rates of the above papers to $O(h)$ in one direction. We also suspect the methods used in our paper could be adapted to the directed last passage percolation problem studied in \cite{calder2015directed}.

We also briefly note that nondominated sorting is equivalent to the problem of finding the length of a longest chain (i.e. a totally ordered subset) in a partially ordered set, which is a well-studied problem in the combinatorics and probability literature \cite{ulam1961, hammersley1972, bollobas1988, deuschel1995}. In particular, $U_n(x)$ is equal to the length of a longest chain in $X$ consisting of points less than $x$ in the partial order.

\section{Main results}\label{sec:results}
We begin by introducing definitions and notation that will be used throughout the paper. In our results and proofs, we let $C$ denote a constant that does not depend on any other quantity, and $C_k$ denotes a constant dependent on the variable $k$. Be advised that the precise value of constants may change from line to line. To simplify the proofs, we model the data using a Poisson point process. Given a nonnegative function $\rho \in L^1(\R^d)$, we let $X_\rho$ denote a Poisson point process with intensity function $\rho$. Hence, $X_\rho$ is a random, at most countable subset of $\R^d$ with the property that for every Borel measurable set $A\subset \R^d$, the cardinality $N(A)$ of $A \cap X_\rho$ is a Poisson random variable with mean $\int_A \rho \dif{x}$. Given two measurable disjoint sets $A,B  \subset \R^d$, the random variables $N(A)$ and $N(B)$ are independent. Further properties of Poisson processes can be found in \cite{kingman1992poisson}. In this paper we consider a Poisson point process $X_{n\rho}$ where $n\in \N$ and $\rho \in C(\R^d)$ satisfies
\begin{align}\label{eq:rhobounds}
    0 < \rhomin \leq \rho \leq \rhomax.
\end{align}
We denote by $C_\rho$ a constant depending on $\rhomin$ and $\rhomax$, and possibly also on $[\rho]_{C^{0,1}(\R^d)}$ and $\norm{D^2 \rho}_{L^\infty(\R^d)}$ in those results that assume $\rho \in \Czeroone$ and $\rho \in \Ctwo$ respectively. Given $R\geq 0$, we define
\begin{equation*}
    \Omega_{R} = \set{x\in [0,1]^d: (x_1\ldots x_d)^{1/d} > R}
\end{equation*}
and
\begin{equation*}
    \partial_{R} \Omega = \set{x\in [0,1]^d: (x_1\ldots x_d)^{1/d} = R}.
\end{equation*}
Let $u$ denote the viscosity solution of \eqref{eq:PDEaux}. Given a finite set $A \subset \R^d$, let $\ell(A)$ denote the length of the longest chain in the set $A$. Given a domain $\Omega \subset \R^d$, the Pareto-depth function $U_n$ in $\Omega$ is defined by
\begin{equation*}
    U_n(x) = \ell([0,x] \cap X_{n\rho} \cap \Omega)
\end{equation*}
where $[0,x] := [0,x_1]\times \ldots \times [0,x_d]$. The scaled Pareto-depth function is defined by 
\begin{align}\label{eq:scaledDepth}
    u_n(x) = \frac{d}{c_d} n^{-1/d}U_n(x)
\end{align}
where $c_d$ is the constant defined by
\begin{align}\label{eq:definitionofcd}
    c_d = \lim_{n\to \infty} n^{-1/d}\ell([0,1]^d \cap X_n) \text{ a.s. }
\end{align}
For a subset $S \subset \R^d$ , we write $u_n(S)$ to denote $\frac{d}{c_d} n^{-1/d}\ell(S \cap X_{n\rho})$. This particular scaling is chosen to eliminate the constant on the right-hand side of \eqref{eq:PDEsemi}.
\begin{remark}
There are several results regarding the constant $c_d$ that have been established in the literature. Hammersley showed that $\lim_{n\to \infty} n^{-1/2}\ell(X_n \cap [0,1]^2) = c \text{ a.s. }$ and conjectured that $c = 2$ in \cite{hammersley1972}. In subsequent works, Logan and Shepp \cite{logan1977} and Vershik and Kerov \cite{vershik1977} showed that $c\geq 2$ and $c\leq 2$. The exact values of $c_d$ for $d > 2$ remain unknown, although Bollob\'as and Winkler showed in \cite{bollobas1988} that
\[
\frac{d^2}{d!^\frac{1}{d} \Gamma\left(\frac{1}{d}\right)} \leq c_d < e {\rm \ \ for \ all \ }d\geq 1.
\]
\end{remark}
Now we state our main convergence rate results. Let $u_n$ denote the Pareto-depth function in $\Omega_R$ and let $u$ denote the viscosity solution of \eqref{eq:PDEaux}.
\begin{theorem}\label{thm:mainaux}
Given $k\geq 1$ and $\rho \in C^{0,1}(\R^d)$ satisfying \eqref{eq:rhobounds}, the following statements hold.
\begin{enumerate}
\item[(a)]
Given $R\in (0,1]$, and $n^{1/d} \geq C_{d,k,\rho} R^{-(2d^2 - d - 1)}$ we have
\begin{align*}
\p \left( \sup_{\Omega_R} \left(u_n - u \right) > C_{d,\rho, k} R^{\frac{-2d^2 + d + 1}{4}} n^{-1/4d} \left(\frac{\log^{2} n}{  \log \log n} \right)^{1/2} \right) \leq C_{d,\rho, k} R^{-C_d}  n^{-k}.
\end{align*}

\item[(b)]
Assume $\rho \in C^2(\R^d_+)$. Then there exists $C_d>0$ such that for all $R\in (0,C_d)$ and $n^{1/d} \geq C_{d,k,\rho} R^{-2d^2 + 4d - 4}\log(n)^C$ we have
\begin{align*}
\p \left( \sup_{\Omega_R} \left(u - u_n \right) > C_{d,\rho, k} R^{\frac{-2d^2 + d}{3}} n^{-1/3d} \left(\frac{\log^{2} n}{  \log \log n} \right)^{2/3} \right) \leq C_{d,\rho, k} R^{-C_d} n^{-k}.
\end{align*}
\end{enumerate}
\end{theorem}
Theorem \ref{thm:mainaux} depends on the parameters $R$ and $k$. Although $R$ is a constant in this result, we have stated the explicit dependence on $R$ as it is required to extend the rates from $\Omega_R$ to $\Omega_0$. Observe that the convergence rates become trivial as $R \to 0^{+}$, as the proof makes use of estimates for the Lipschitz constant and semiconvexity constant of $u$ on $\Omega_R$ that blowup as $R$ tends to $0$. Also observe that the convergence rate in (b) is sharper than in (a), thanks to our use of the semiconvexity estimates established in Theorem \ref{thm:mainsemiconvexity}. Let $v$ denote the solution of 
\begin{equation}\label{eq:PDEfull}
\left\{\begin{aligned}
v_{x_1}v_{x_2}\ldots v_{x_d}  &=  \rho&&\text{in }\Omega_0\\ 
v &= 0 &&\text{on } \partial_0 \Omega
\end{aligned}\right.
\end{equation}
In the next result we state our convergence rates on $\Omega_0 = [0,1]^d$ which are proved by using $u$ as an approximation to $v$ and setting $R$ equal to the optimal value that balances the approximation error term with the convergence rate. Let 
\begin{align}\label{eq:scaledDepthFull}
    v_n(x) = \frac{d}{c_d} n^{-1/d}\ell(X_{n \rho} \cap [0,x])
\end{align}
denote the scaled Pareto-depth function in $[0,1]^d$.
\begin{theorem}\label{thm:mainfull}
Given $k\geq 1$ and $\rho \in \Czeroone$ satisfying \eqref{eq:rhobounds}, the following statements hold.
\begin{enumerate}
\item[(a)]
For all $n > C_{d,k,\rho}$ we have
\begin{align*}
\p \left( \sup_{\Omega_0} \left(v_n - v \right) > n^{-1/(2d^3 + d^2 + 5d + 1)} \right) \leq C_{d,k,\rho} n^{-k}.
\end{align*}

\item[(b)]
Assume $\rho \in \Ctwo$. Then for all $n > C_{d,k,\rho}$ we have
\begin{align*}
\p \left( \sup_{\Omega_0} \left(v - v_n \right) >  n^{-1/(2d^3 - d^2 + 3d + 1)} \right) \leq C_{d,k,\rho} n^{-k}.
\end{align*}
\end{enumerate}
\end{theorem}
Observe that the rate in (b) is sharper thanks to the sharper one-sided rate in Theorem \ref{thm:mainfull}. We do not know for certain whether the rates in Theorem \ref{thm:mainaux} and \ref{thm:mainfull} are optimal, although it seems likely that they are not.

These results also extend to the situation when $X_{n\rho} = \set{Y_1,\ldots ,Y_n}$ where $Y_1,\ldots ,Y_n$ are $\text{i.i.d.}$ random variables with continuous density $\rho$. The analogues of Theorems \ref{thm:mainaux} and \ref{thm:mainfull} in this context follow from Lemma \ref{lem:iidvariablebound}.
\begin{corollary}
Let $Y_1,\ldots ,Y_n$ be $\text{i.i.d.}$ random variables with density $\rho$. Then Theorems \ref{thm:mainaux} and \ref{thm:mainfull} hold when $X_{n\rho} = \set{Y_1, \ldots , Y_n}$.
\end{corollary}

A key step in our proof of the sharper one-sided rate is a quantitative estimate on the semiconvexity constant of $u$. As the Hamiltonian $H(p) = (p_1\ldots p_d)^{1/d}$ is concave, the results on semiconvex viscosity solutions in \cite{bardi1997} would lead us to suspect that $u$ is semiconvex. However, from an examination of the function $w(x) = d(x_1\ldots x_d)^{1/d}$ that solves (\ref{eq:PDE}) with $\rho=1$, it is evident that solutions of (\ref{eq:PDE}) on $\R^d_+$ need not be semiconvex nor semiconcave due to the gradient singularity on the coordinate axes. This motivates us to determine the rate at which the semiconvexity constant of $u$ on $\Omega_R$ blows up as $R \to 0^{+}$. For proving these results it is convenient to raise the PDE to the $1/d$ power and pose the Dirichlet problem on the more general domains $\Omega_{R,M} = \set{x\in [0,M]^d: (x_1\ldots x_d)^{1/d} > R}$ with boundary conditions on $\partial_{R,M} \Omega = \set{x\in [0,M]^d: (x_1\ldots x_d)^{1/d} = R}$. Let $R>0$, $M\geq 1$, and let $u$ denote the solution of
\begin{equation}\label{eq:PDEsemi}
\left\{\begin{aligned}
(u_{x_1}u_{x_2}\ldots u_{x_d})^{1/d}  &=  \rho^{1/d}&&\text{in }\Omega_{R,M}\\
u &= 0 &&\text{on } \partial_{R,M} \Omega.
\end{aligned}\right.
\end{equation}
Our result on semiconvexity bounds the rate at which the semiconvexity constant of $u$ on $\Omega_{R,M}$ blows up as $R$ tends to $0$. This result enables us to establish the sharpened one-sided convergence rates in case (b) of Theorem \ref{thm:mainaux} and Theorem \ref{thm:mainfull}.
\begin{theorem}\label{thm:mainsemiconvexity}
    Let $u$ denote the solution to (\ref{eq:PDEsemi}). Then there exists a constant $C_\rho > 0$ such that for all $R \leq C_\rho$, $x\in \Omega_{R,M}$, and $h\in \R^d$ such that $x\pm h \in \Omega_{R,M}$ we have
    \begin{align*}
        u(x+h) - 2u(x) + u(x-h) \geq -C_{d,M,\rho} R^{-2d+1} \abs{h}^2
    \end{align*}
    where 
    \begin{align*}
        C_{d,M,\rho} = C_d(1 + M \rhomax^{1/d}) \left(\rhomin^{-(d-1)/d}\norm{D\rho}_{L^\infty(\partial_{R,M}\Omega)} M^{2d-1} + \rhomax^{1/d} M^{2d-2} \right).
    \end{align*}
\end{theorem}
\subsection{Definition of Viscosity Solution}
Here we briefly state for reference the definition of viscosity solution for the first-order equation
\begin{equation}\label{eq:firstorderequation}
    H(Du,u,x) = 0 \text{ in } \O,
\end{equation}
where $H$ is continuous and $\O \subset \R^d$.
\begin{definition}[Viscosity solution]
We say that $u\in \USC(\O)$ is a \textit{viscosity subsolution} of \eqref{eq:firstorderequation} if for every $x\in \O$ and every $\varphi \in C^\infty(\R^d)$ such that $u-\varphi$ has a local maximum at $x$ with respect to $\O$ we have
\begin{equation*}
    H(D\varphi(x), \varphi(x), x) \leq 0.
\end{equation*}
We will often say that $u\in \USC(\O)$ is a viscosity solution of $H \leq 0$ in $\O$ when $u$ is a viscosity subsolution of \eqref{eq:firstorderequation}. Similarly, we say that $u\in \LSC(\O)$ is a \textit{viscosity supersolution} of \eqref{eq:firstorderequation} if for every $x\in \O$ and every $\varphi \in C^\infty(\R^d)$ such that $u-\varphi$ has a local minimum at $x$ with respect to $\O$ we have
\begin{equation*}
    H(D\varphi(x), \varphi(x), x) \geq 0.
\end{equation*}
We also say that $u\in \LSC(\O)$ is a viscosity solution of $H \geq 0$ in $\O$ when $u$ is a viscosity supersolution of \eqref{eq:firstorderequation}. Finally, we say that $u$ is a \textit{viscosity solution} of \eqref{eq:firstorderequation} if $u$ is both a viscosity subsolution and a viscosity supersolution.
\end{definition}
\subsection{Outline of Proof of Theorem \ref{thm:mainaux}}

\begin{figure}[!t]
\centering
\subfloat[$A_n\subset A$]{\includegraphics[width=.45\textwidth]{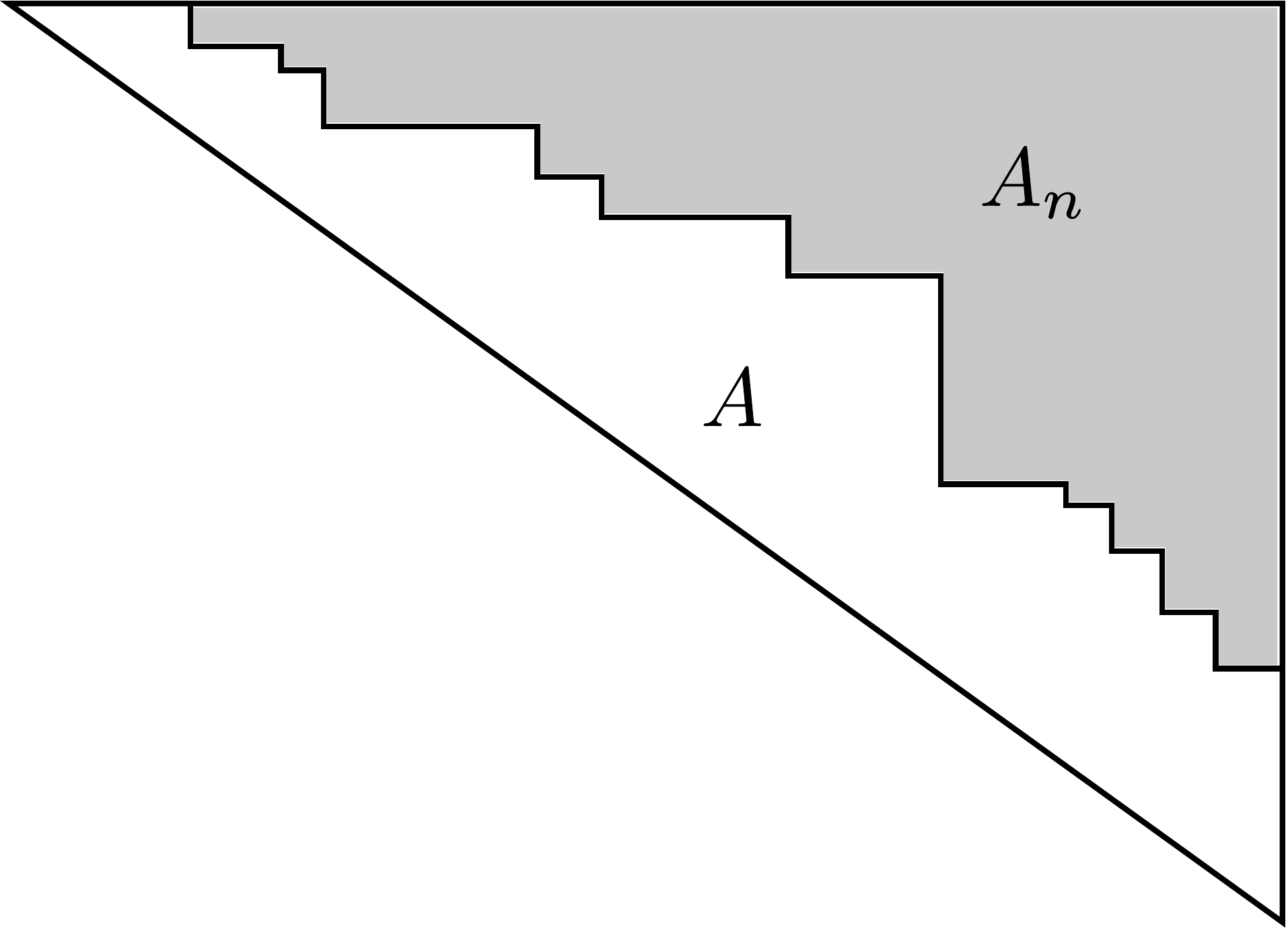}}\hfill
\subfloat[$B\subset B_n$]{\includegraphics[width=.45\textwidth]{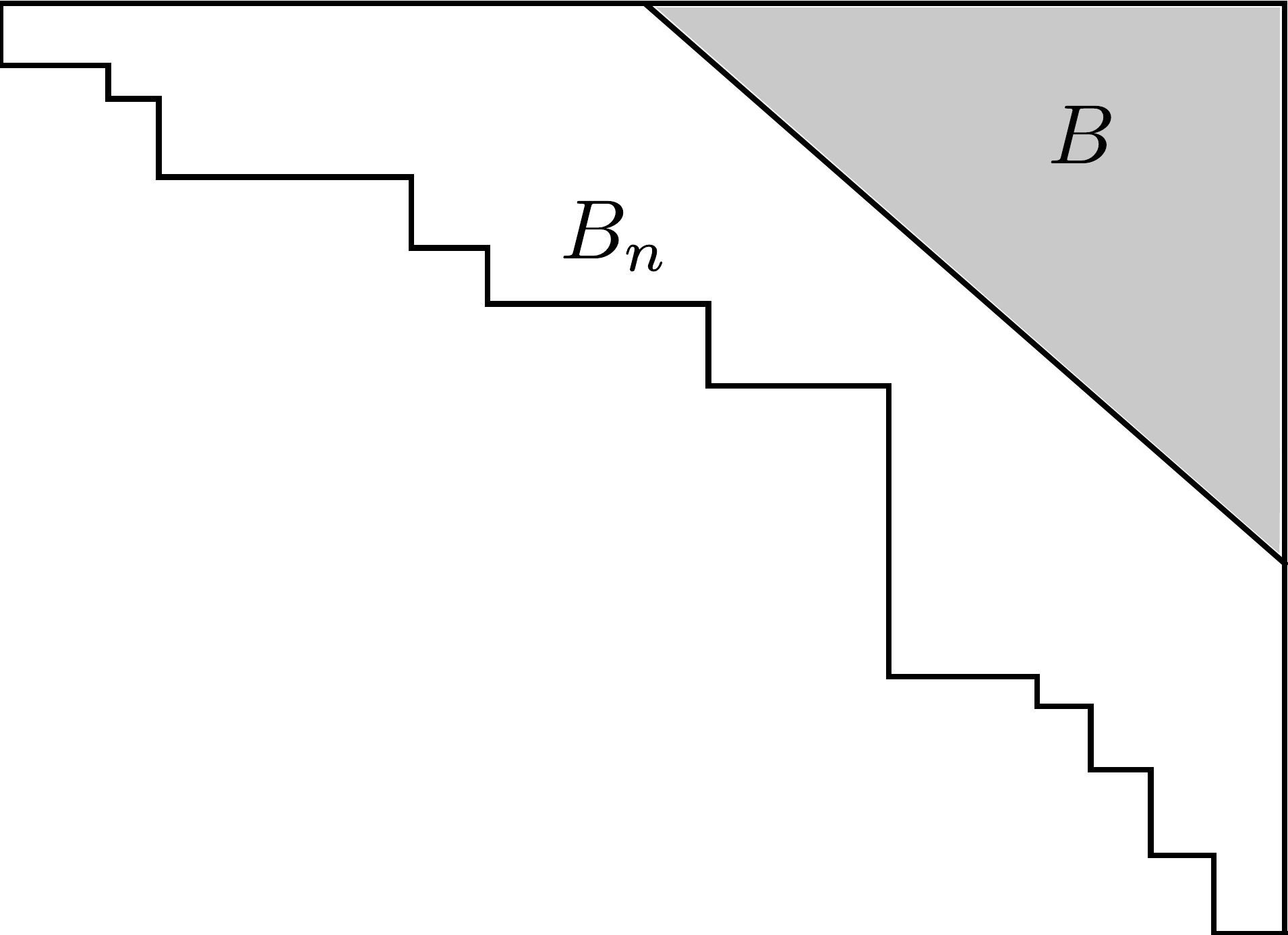}}
\caption{Illustration of the sets $A_n,B_n$ and $A,B$ used in the proof outline, and the viscosity touching property that $A_n\subset A$ and $B\subset B_n$. }
\label{fig:AB}
\end{figure}

Here, we present a high-level outline of the proof of Theorem \ref{thm:mainaux}.     The proof follows a stochastic homogenization argument, similar to \cite{armstrong2014error}, but with different ingredients. We first study the asymptotics of the longest chain in orthogonal simplices of the form
\begin{equation}\label{eq:gen_simplex}
S_{y,p} := \set{x\in (-\infty,y]^d: 1 + (x-y)\cdot p^{-1} \geq 0}
\end{equation}
and
\begin{equation}\label{eq:gen_simplex2}
    S_{p} := \set{x\in (-\infty,0]^d: 1 + x\cdot p^{-1} \geq 0}
\end{equation}
where $p\in (0,\infty)^d$ and $p^{-1} = (p_1^{-1},\ldots ,p_d^{-1})$. The set $S_{p}$ is an \emph{orthogonal simplex} with side length $p_i$ in the $i^{\rm th}$ coordinate direction. The measure of $S_p$ is given by
\begin{equation}\label{eq:measureSp}
|S_p| = \frac{p_1 \cdots p_d}{d^d}.
\end{equation}
The sets $A$ and $B$ in Figure \ref{fig:AB} show examples of orthogonal simplices. The longest chain in an orthogonal simplex, $u_n(S_p)$, can be thought of as a \emph{cell problem} from homogenization, in the sense that it is a simpler local problem, whose solution allows us to prove our main results. The value of $p$ will turn out to be proportional to the gradient $Du$ of the continuum limit $u$, as in homogenization, and the cell problem exactly describes the local behaviour of $u_n$ for large $n$.

For simplicity, we will take the intensity $\rho$ to be constant on $\R^d$ throughout the rest of this section, and we denote the constant value by $\rho>0$. The extension to nonconstant intensities follows by approximating $\rho$ from above and below by constant intensities on the simplices $S_p$, which are vanishingly small as $n\to \infty$. It was shown in \cite{calder2016direct} that 
\begin{equation}\label{eq:splim}
\lim_{n\to \infty}u_n(S_p) = d\rho^{1/d}|S_p|^{1/d},
\end{equation}
with probability one. This is proved by reducing to the unit cell problem $u_n(S_{\mathds{1}})$ using \emph{dilation invariance} of $u_n$ and the sets $S_p$. In particular, if $\Phi:\R^n\to \R^n$ is any dilation (i.e., $\Phi x = (a_1x_1,\dots,a_dx_d)$ for $a_i>0$), then we have $\Phi S_p = S_{\Phi^{-1} p}$ and so 
\[ \ell(X_{n\rho}\cap S_p) = \ell(\Phi X_{n\rho}\cap \Phi S_p) \sim \ell(X_{n|\Phi|^{-1}\rho}\cap \Phi S_p).\]
We then choose $\Phi$ so that $\Phi S_p = S_{\mathds{1}}$, that is $a_i=p_{i}$, to obtain 
\[u_n(S_p) \sim (p_1\cdots p_d)^{-1/d}u_{n|\Phi|^{-1}}(S_{\mathds{1}})\]
This shows that the scaling limit \eqref{eq:splim} for a general simplex $S_p$ follows directly follow from one for the unit simplex $S_{\mathds{1}}$. 

The first ingredient in our proof is a convergence rate, with high probability, for the cell problem \eqref{eq:splim}. In particular, in Theorem \ref{thm:simplexrate} we improve \eqref{eq:splim} by showing that
\begin{equation}\label{eq:cellprob}
 u_n(S_p) = d\rho^{1/d}|S_p|^{1/d} + O\left( n^{-1/2d}|S_p|^{1/2d}\right)
 \end{equation}
with high probability, up to logarithmic factors. The proof is based on the concentration of measure results in \cite{bollobas1992height} for the length of a longest chain in boxes, which uses Azuma's inequality. We adapt these results to the simplices $S_p$.

To illustrate how the cell problem \eqref{eq:cellprob} is used to prove our main results, let $x_0 \in \Omega_R$ and define
\[A_n = \left\{x\in [0,x_0]\cap \Omega_R\, : \,  u_n(x_0) - u_n(x)\leq \vare \right\}.\]
Basically by definition  we have $u_n(A_n)\approx \vare$ (see Lemma \ref{lem:chaininAn} for a precise statement of this). Now, if $u_n$ is well approximated by a smooth function $u$, then we can Taylor expand $u$ to show that $A_n\approx x_0 + S_p$ where $p^{-1} = \vare^{-1}Du(x_0)$. In this case we use \eqref{eq:measureSp} to obtain
\[|S_p| = \frac{p_1\cdots p_d}{d^d} = \frac{\vare^d}{d^d u_{x_1}(x_0)\cdots u_{x_d}(x_0)},\]
and hence
\begin{equation}\label{eq:heuristic}
\vare\approx u_n(A_n) \approx u_n(S_p) \approx d\rho^{1/d}|S_p|^{1/d} = \frac{\vare \rho^{1/d}}{(u_{x_1}(x_0)\cdots u_{x_d}(x_0))^{1/d}}.
\end{equation}
Rearranging we obtain the Hamilton-Jacobi equation \eqref{eq:PDE}.

The proof of our main result involves keeping track of the error estimate from the cell problem convergence rate \eqref{eq:cellprob} in the argument above, as well as using the viscosity solution framework to push the Taylor expansion arguments above onto smooth test functions. For this, we use the fact that the set $A_n$ satisfies a \emph{viscosity property}. That is, if $u_n - \varphi$ attains its maximum at $x_0$, then
\[u_n(x) - \varphi(x) \leq u_n(x_0) - \varphi(x_0),\]
and so
\[\varphi(x_0) - \varphi(x) \leq u_n(x_0) - u_n(x).\]
It follows that $A_n\subset A$, where $A$ is the corresponding set defined for the test function $\varphi$, given by
\[A = \left\{x\in [0,x_0]\cap \Omega_R\, : \,  \varphi(x_0) - \varphi(x)\leq \vare \right\}.\]
The inclusion $A_n\subset A$ is depicted in Figure \ref{fig:AB} (A). Then \eqref{eq:heuristic} is modified by inserting the inequality $u_n(A_n) \leq u_n(A)$, and then approximating $A$ by a simplex, which is possible when the test function $\varphi$ is sufficiently smooth. This gives a rate in only one direction, since we get a subsolution condition, and so we also need to consider touching from below; that is, examining the minimum value of  $u-\varphi$. In this case the inequalities are reversed and we have $A \subset A_n$. This inclusion is depicted in Figure \ref{fig:AB} (B), where we write $B$ and $B_n$ in place of $A$ and $A_n$ (different names are used in the proofs of our main results for technical reasons).

The convergence rates in our main results are then proved using a maximum principle argument, which examines the maximum of $u_n-u$ (and subsequently $u - u_n$)  and uses the viscosity properties and cell problem convergence rates described above. In the case where $u$ is a non-smooth viscosity solution, one typically replaces $u$ by smoother approximate sub-and super-solutions obtained by inf- and sup-convolutions, to allow for Taylor expansions (equivalently we may use a doubling variables argument). Another main contribution of our paper is a new semiconvexity estimate for the  solution $u$ of \eqref{eq:PDEaux} on the rounded off domain $\Omega_R$ (see Theorem \ref{thm:mainsemiconvexity}). We sharply characterize the blow-up of the gradient and semiconvexity constant of $u$ as $R\to 0$. This allows us to avoid the sup-convolution and use $\varphi=u$ directly in the maximum principle argument when bounding $u-u_n$. This leads to the better $O(n^{-1/3d})$ convergence rate in Theorem \ref{thm:mainaux} (b). In the other direction, when bounding $u_n-u$, we would need semiconcavity of $u$, which is not true in general, so we use the inf-convolution to produce a semiconcave approximation, leading to the worse $O(n^{-1/4d})$ rate in Theorem \ref{thm:mainaux} (a). As $R\to 0$ and we approach the corner singularity problem \eqref{eq:PDE}, we lose control of the semiconvexity estimates, and the solution of \eqref{eq:PDE} is neither semiconvex nor semiconcave in general. We thus obtain the rates in Theorem \ref{thm:mainfull} by approximation to the rounded off case \eqref{eq:PDEaux}, leading to substantially worse rates of convergence in the presence of the corner singularity in \eqref{eq:PDE}. 

While our proof techniques are at a high level similar to \cite{armstrong2014error}, the details are substantially different and cannot be compared directly. We can, however, compare the final convergence rates we obtain. In \cite{armstrong2014error} the authors consider stochastic homogenization of Hamilton-Jacobi equations of the form
\[u_t^\vare + H\left(Du^\vare,\frac{x}{\vare},\omega\right) = 0 \ \ \text{in } \R^d\times (0,\infty),\]
and obtain quantitative homogenization rates of
\begin{equation}\label{eq:armstrong}
-O\left( \vare^{1/8 - \delta}\right) \leq u^\vare - u \leq O\left( \vare^{1/5 - \delta}\right),
\end{equation}
for any $\delta>0$, in the setting where $H$ is level-set convex and coercive in the gradient. Our Hamiltonian $H(p)=p_1\cdots p_d$ is level-set concave (and in fact we can write it as $H(p)=(p_1\cdots p_d)^{1/d}$ to obtain a concave Hamiltonian), but it is not coercive. Recalling that nondominated sorting can be viewed as a stochastic Hamilton-Jacobi equation \eqref{eq:UnPDE} with rapidly oscillating terms on the order of $\vare=n^{-1/d}$ we see that our rates in Theorem \ref{thm:mainaux} yield
\[-C_1 \vare^{1/3} \leq u_n - u \leq C_2 \vare^{1/4},\]
up to logarthmic factors, which are substantially sharper than \eqref{eq:armstrong}. 

\subsection{Outline of Paper}
Here we outline the remainder of the paper. In Section \ref{sec:prelim} we establish a maximum principle and Lipschitz estimates for \eqref{eq:PDEaux} that are used throughout the paper. In Section \ref{sec:longestchain} we extend the work of Bollob\'as and Winkler in \cite{bollobas1992height} and establish rates of convergence for the longest chain problem in simplices. In Section \ref{sec:supporting} we establish our principle lemma for proving Theorem \ref{thm:mainaux}, which shows for a strict supersolution $\varphi$ of (\ref{eq:PDEaux}) that the maximum of $u_n - \varphi$ occurs near the boundary with high probability. In Section \ref{sec:mainrates} we present the proofs of Theorems \ref{thm:mainaux} and \ref{thm:mainfull}, and in Section \ref{sec:semiconvexity} we present the proof of Theorem \ref{thm:mainsemiconvexity}.

\section{Maximum Principle and Lipschitz estimates}\label{sec:prelim}
In this section we establish fundamental results regarding the PDE \eqref{eq:PDE} that are used throughout the paper. First we show that if $u$ satisfies $u_{x_1}\ldots u_{x_d} = \rho$ on a domain $\Omega$, then a closely related PDE is also automatically satisfied at certain boundary points. Given $M > 0$, let $\Omega \subset [0,M]^d$ and define
\begin{align}\label{eq:partialstaromega}
    \partial^{\ast} \Omega = \set{y\in \ov{\Omega}: y_i = M \text{ for some } i \text{ and } \exists \vare>0 \text{ such that } B(y,\vare) \cap [0,M)^d \subset \Omega}.
\end{align}
\begin{lemma}\label{lem:viscsolonboundary}
Given $\Omega \subset [0,M]^d$, let $\partial^\ast \Omega$ be given by \eqref{eq:partialstaromega} and let
$\rho \in C(\ov{\Omega})$ satisfy \eqref{eq:rhobounds}. Then the following statements hold.
\begin{enumerate}
    \item[(a)]
    Suppose that $u$ satisfies $u_{x_1}u_{x_2}\ldots u_{x_d}  \leq  \rho$ in $\Omega$. Then $u$ satisfies 
  \begin{align*}
      \prod_{i=1}^d (u_{x_i})_+  \leq  \rho \text{ in } \Omega \cup \partial^{\ast} \Omega.
  \end{align*}
    \item[(b)]
    Suppose that $u$ satisfies $u_{x_1}u_{x_2}\ldots u_{x_d}  \geq  \rho$ in $\Omega$ and $u$ is nondecreasing in each coordinate. Then $u$ satisfies 
    \begin{align*}
        \prod_{i=1}^d (u_{x_i})_+  \geq  \rho \text{ in }  \Omega \cup \partial^{\ast} \Omega.
    \end{align*}
\end{enumerate}
\end{lemma}
\begin{proof}
To prove (b), let $\psi \in C^\infty(\R^d)$ such that $u - \psi$ has a local minimum at $x_0\in \Omega$, and show that $\psi_{x_i}(x_0)\geq 0$. For $y$ in a neighborhood of $x_0$ we have
\begin{align*}
    u(x_0) - u(y) \leq \psi(x_0) - \psi(y).
\end{align*}
Since $u$ is nondecreasing in each coordinate, when $h>0$ is sufficiently small we have
\begin{align*}
    0\leq \frac{u(x_0) - u(x_0 - he_i)}{h} \leq \frac{\psi(x_0) - \psi(x_0 - he_i)}{h}.
\end{align*}
Hence, $\psi_{x_i}(x_0) \geq 0$. Now let $x_0 \in \Omega \cup \partial^{\ast} \Omega$ and let $\varphi \in C^\infty(\R^d)$ such that $u- \varphi$ has a local minimum at $x_0$. Without loss of generality, we may assume that $u-\varphi$ attains a strict global minimum at $x_0$. If $x_0 \in \Omega$, then $\varphi_{x_i}(x_0) \geq 0$, and we have
\begin{align*}
     \prod_{i=1}^d (\varphi_{x_i}(x_0))_+ = \prod_{i=1}^d \varphi_{x_i}(x_0) \geq \rho(x_0).
\end{align*}
If $x_0 \in \partial^{\ast} \Omega$, let $\varphi_\vare(x) = \varphi(x) - \vare \sum_{i=1}^d \frac{1}{M-x_i}$, and we claim that $u - \varphi_\vare$ attains its minimum over $\ov{\Omega}$ in $\ov{\Omega} \cap [0,M)^d$. To prove this, let $y_k \in [0,M)^d$ be a minimizing sequence. Replacing $y_k$ with a convergent subsequence, we may assume that $y_k \to y \in [0,M]^d$. It is clear from the definition of $\varphi_\vare$ that we must have $y \in [0,M)^d$. There exist sequences $\vare_k \to 0$ and $x_k \to x_0$ such that $\vare_k > 0$ and $u - \varphi_{\vare_k}$ has a local minimum at $x_k \in [0,M)^d$. Since $x_0 \in \partial^{\ast} \Omega$, there exists $N > 0$ such that
$x_k \in \Omega$ for $k > N$. Hence, for all $k > N$ we have
\begin{align*}
    \prod_{j=1}^d \left(\varphi_{x_j}(x_k) - \frac{\vare_k}{(M-x_{k,j})^2}   \right) \geq \rho(x_k).
\end{align*}
Since $u - \varphi_{\vare_k}$ has a local minimum at $x_0$ and $u$ is nondecreasing in each coordinate, we have $(\varphi_{\vare_k})_{x_j} = \varphi_{x_j}(x_k) - \frac{\vare_k}{(M-x_{k,j})^2} \geq 0$ for $j=1,\ldots ,d$. Hence for $k > N$ we have
\begin{align*}
    \prod_{j=1}^d \left(\varphi_{x_j}(x_k)  \right)_+ = \prod_{j=1}^d \left(\varphi_{x_j}(x_k)  \right) \geq \rho(x_k).
\end{align*}
Letting $k\to \infty$, we have $\prod_{j=1}^d \left(\varphi_{x_j}(x_0)  \right)_+ \geq \rho(x_0)$. To prove (a), let $x_0 \in \Omega \cup \partial^{\ast} \Omega$, and let $\varphi \in C^{\infty}(\R^d)$ such that $u-\varphi$ has a local maximum at $x_0$. If $\varphi_{x_i}(x_0) \leq 0$ for some $1\leq i \leq d$, then we have
\begin{align*}
    0 = \prod_{j=1}^d \left(\varphi_{x_j}(x_0)  \right)_+ \leq \rho(x_0).
\end{align*}
Assume that $\varphi_{x_i}(x_0) > 0$ for each $i$. If $x_0 \in \Omega$, then we have
\begin{align*}
    \prod_{j=1}^d \varphi_{x_j}(x_0)   = \prod_{j=1}^d \left(\varphi_{x_j}(x_0)  \right)_+ \leq \rho(x_0).
\end{align*}
If $x_0 \in \partial^{\ast} \Omega$. Without loss of generality, we may assume that $u-\varphi$ attains a strict global maximum at $x_0$. Let $\varphi_\vare(x) = \varphi(x) + \vare \sum_{i=1}^d \frac{1}{M-x_i}$. As in (a), $u - \varphi_\vare$ attains its maximum over $\ov{\Omega}$ in $\ov{\Omega} \cap [0,M)^d$. Hence, there exist sequences $\vare_k \to 0$ and $x_k \to x_0$, $x_k \in [0,M]^d$ such that $u - \varphi_{\vare_k}$ has a local maximum at $x_k \in [0,M)^d$. Then when $k$ is large we have $x_k \in \Omega$, hence
\begin{align*}
    \prod_{j=1}^d \varphi_{x_j}(x_k)_+ \leq \prod_{j=1}^d \left(\varphi_{x_j}(x_k) + \frac{\vare_k}{(M-x_{k,j})^2}   \right) \leq \rho(x_k).
\end{align*}
Since $\varphi$ is smooth, we have $\varphi_{x_j}(x_k) > 0$ for $k$ sufficiently large. Letting $k\to \infty$, we have
\begin{align*}
    \prod_{j=1}^d \left(\varphi_{x_j}(x_0)_+  \right) \leq \rho(x_0).
\end{align*}

\end{proof}
Next we establish that subsolutions and supersolutions of \eqref{eq:PDE} may be perturbed to strict subsolutions and supersolutions. Let $L$ and $H$ be given by $L(p) = (p_1\ldots p_d)^{1/d}$, $H(p) = p_1\ldots p_d$.
\begin{proposition}\label{prop:perturb}
Given $V \subseteq \R^d$, let $\rho \in C(\ov{V})$ satisfy \eqref{eq:rhobounds}. Then the following statements hold.
\begin{enumerate}
\item[(a1)]
Let $u$ satisfy $L(Du) \geq \rho$ on $V$. Then for all $\lambda > 0$ we have
\begin{align*}
L(D((1+\lambda)u)) \geq \rho + \rhomin \lambda \text{ on } V.
\end{align*}

\item[(b1)]
Let $u$ satisfy $L(Du)\leq \rho$ on $V$. Then for all $\lambda \in (0,1]$ we have
\begin{align*}
L(D((1-\lambda)u)) \leq \rho - \rhomin \lambda \text{ on } V.
\end{align*}

\item[(a2)]
Let $u$ satisfy $H(Du) \geq \rho$ on $V$. Then for all $\lambda > 0$ we have
\begin{align*}
H(D((1+\lambda)u)) \geq \rho + d\rhomin \lambda \text{ on } V.
\end{align*}

\item[(b2)]
Let $u$ satisfy $H(Du)\leq \rho$ on $V$. Then for all $\lambda \in (0,1]$ we have
\begin{align*}
H(D((1-\lambda)u)) \leq \rho - \rhomin \lambda \text{ on } V.
\end{align*}
\end{enumerate}
\end{proposition}
\begin{proof}
To prove (a1), let $x\in V$. Then there exists $\varphi \in C^\infty(\R^d)$ such that $u-\varphi$ has a local minimum at $x$. Consequently, $(1+\lambda)u - (1+\lambda)\varphi$ has  a local minimum at $x$, so 
\begin{equation*}
L((1+\lambda)D\varphi(x)) = (1+\lambda) f(x) \geq f(x) + \lambda (\inf_V f)
\end{equation*}
and the statement follows. The proofs of the other statements are very similar and omitted here, making use of the inequalities $(1+\lambda)^d \geq (1+d\lambda)$ in (a2) and $(1-\lambda)^d \leq (1-\lambda)$ in (b2).
\end{proof}
Now we establish a comparison principle for the PDE (\ref{eq:PDE}).
\begin{theorem}\label{thm:maximumprinciple}
Given $\Omega \subset [0,M]^d$, let $\Gamma \subset \ov{\Omega}$ be a closed set such that $ \ov{\Omega} \subseteq \Gamma \cup \Omega \cup \partial^{\ast} \Omega$, where $\partial^\ast \Omega$ is given by \eqref{eq:partialstaromega}. Suppose that $\rho \in C(\ov{\Omega})$ satisfies \eqref{eq:rhobounds}, and $u\in C(\ov{\Omega})$ and $v\in C(\ov{\Omega})$ satisfy
\begin{equation}
\left\{\begin{aligned}
u_{x_1}u_{x_2}\ldots u_{x_d}  &\leq  \rho&&\text{in }\Omega\\ 
u &= g_1 &&\text{on } \Gamma,
\end{aligned}\right.
\end{equation}
and
\begin{equation}
\left\{\begin{aligned}
v_{x_1}v_{x_2}\ldots v_{x_d}  &\geq  \rho &&\text{in }\Omega\\ 
v &= g_2 &&\text{on } \Gamma,
\end{aligned}\right.
\end{equation}
respectively. Assume that $v$ is nondecreasing in each coordinate and $g_1 \leq g_2$ on $\Gamma$. Then $u \leq v \text{ on } \ov{\Omega}$.
\end{theorem}
\begin{proof}
Given $\lambda \in (0,1)$, set $v_\lambda = (1+\lambda)v$, and suppose for contradiction that $\sup_{\ov{\Omega}} (u-v_\lambda) > 0$. Let
\begin{equation*}
    \Phi(x,y) = u(x) - v_\lambda(y) - \frac{\alpha}{2}\abs{x-y}^2.
\end{equation*}
Then $\Phi \in C(\ov{\Omega} \times \ov{\Omega})$ and $\ov{\Omega}$ is bounded. Hence $\Phi$ attains its maximum at some $(x_\alpha, y_\alpha) \in \ov{\Omega}\times \ov{\Omega}$. Then we have
\begin{equation*}
    \Phi(x_\alpha, y_\alpha) \geq \sup_{\ov{\Omega}}(u-v_\lambda) > 0.
\end{equation*}
As $u$ and $-v_\lambda$ are bounded above on $\ov{\Omega}$ we have
\begin{equation}\label{eq:xalphayalphabound}
    \abs{x_\alpha - y_\alpha}^2 \leq \frac{C}{\alpha}.
\end{equation}
As $(x_\alpha, y_\alpha) \in \ov{\Omega}\times \ov{\Omega}$, there exists a sequence $\alpha_n \to \infty$ such that $\set{x_{\alpha_n}}$ and $\set{y_{\alpha_n}}$ are convergent sequences. Letting $x_n = x_{\alpha_n}$ and $y_n = y_{\alpha_n}$, we have $(x_n, y_n) \to (x_0, y_0)$. By \eqref{eq:xalphayalphabound} we have $x_0 = y_0$. By continuity of $\Phi$ we have
\begin{align*}
    \lim_{n \to \infty} \Phi(x_n, y_n) = u(x_0) - v_\lambda(x_0).
\end{align*}
We cannot have $x_0 \in \Gamma$, since $u(x_0) - v_\lambda(x_0) > 0$ and $u \leq v_\lambda$ on $\Gamma$. Hence, $x_0 \in \Omega \cup \partial^\ast \Omega$. As $u - v_\lambda \leq 0$ on $\Gamma$ and $(u-v_\lambda)(x_0) > 0$, by continuity of $u-v_\lambda$ there exists $N > 0$ such that $(x_n,y_n) \in (\Omega \cup \partial^\ast \Omega) \times (\Omega \cup \partial^\ast \Omega)$ for $n > N$. Let $\varphi(x) = \frac{\alpha_n}{2}\abs{x - y_n}^2$ and $\psi(x) = -\frac{\alpha_n}{2}\abs{x_n - y}^2$. Then $u - \varphi$ has a local maximum at $x_n$ and $v_\lambda - \psi$ has a local minimum at $y_n$. Setting $H(p) = p_1\ldots p_d$, Proposition \ref{prop:perturb} gives that $H(Dv_\lambda) \geq \rho + \delta$ on $\Omega$, where $\delta = \lambda \rhomin > 0$. By Lemma \ref{lem:viscsolonboundary} we have $\tilde{H}(Dv_\lambda) \geq \rho + \delta$ and $\tilde{H}(Du) \leq \rho$ on $\Omega \cup \partial^\ast \Omega$. Thus, we have
\begin{align*}
    \tilde{H}(D\varphi(x_n)) &= \tilde{H}(\alpha_n (x_n - y_n)) \leq \rho(x_n)
\end{align*}
and
\begin{align*}
    \tilde{H}(D\psi(y_n)) &= \tilde{H}(\alpha_n (x_n - y_n)) \geq \rho(y_n) + \delta.
\end{align*}
Hence, $\rho(x_n) - \rho(y_n) \geq \delta > 0$, and this gives a contradiction as $n \to \infty$. We conclude that $u \leq v_\lambda = (1+\lambda)v$ on $\ov{\Omega}$. Letting $\lambda \to 0^+$ completes the proof.
\end{proof}
Now we establish estimates on $[u]_{C^{0,1}(\ov{\Omega}_{R,M})}$ with respect to $R$ and $M$. To this end, we state the following theorem, proven in \cite[Theorem 2]{calder2014}. Let $g: \R^d \to [0,\infty)$ be bounded and Borel measurable, and let
\begin{equation}\label{eq:variational}
    U(x) = \sup_{\substack{\gamma \in \mathcal{A} \\ \gamma(1) \leq x}} J(\gamma),
\end{equation}
where 
\begin{equation*}
    J(\gamma) = \int_0^1 g(\gamma(t))^{1/d}\left[\gamma_1'(t)\ldots \gamma_d'(t)  \right]^{1/d}dt,
\end{equation*}
and
\begin{align*}
    \mathcal{A} = \set{\gamma \in C^1([0,1]; \R^d) : \gamma_j^{\prime}(t) \geq 0 \text{ for } j=1,\ldots d }.
\end{align*}
Then the value function $U$ satisfies
\begin{equation}
\left\{\begin{aligned}
U_{x_1}U_{x_2}\ldots U_{x_d}  &=  \frac{1}{d^d}g&&\text{in }\R^d_+\\ 
U &= 0 &&\text{on } \partial \R^d_+.
\end{aligned}\right.
\end{equation}
When $g = \rho \mathbbm{1}_{\Omega_{R,M}}$ and $u$ is given by \eqref{eq:PDEsemi}, Theorem \ref{thm:maximumprinciple} implies that $\frac{1}{d} U = u$ in $\Omega_{R,M}$.
\begin{theorem}\label{thm:lipschitzvar}
Given $\rho \in \Czeroone$ satisfying \eqref{eq:rhobounds}, let $u$ denote the solution to \eqref{eq:PDEsemi}. Then we have
\begin{align*}
     [u]_{C^{0,1}(\ov{\Omega}_{R,M})} \leq  C_d M^{d-1} R^{-(d-1)} \norm{\rho^{1/d}}_{C^{0,1}(\ov{\Omega}_{R,M})}.
\end{align*} 
\end{theorem}
\begin{proof}
Let $U$ be given by \eqref{eq:variational} where $g = \rho \mathbbm{1}_{\Omega_{R,M}}$. In light of the preceding discussion, it is enough to show that
\begin{equation*}
     [U]_{C^{0,1}(\ov{\Omega}_{R,M})} \leq  C_d M^{d-1} R^{-(d-1)} \norm{\rho^{1/d}}_{C^{0,1}(\ov{\Omega}_{R,M})}.
\end{equation*} 
Let $x\in \ov{\Omega}_{R,M}$ and set $f = \rho^{1/d}$. It suffices to show that
\begin{align}
    U(x+he_i) - U(x) \leq C_d h M^{d-1} R^{-(d-1)} \norm{f}_{C^{0,1}(\ov{\Omega}_{R,M})}
\end{align}
when $1\leq i \leq d$ and  $h>0$ is sufficiently small. Given $\vare > 0$, let $\gamma \in \mathcal{A}$ such that $\gamma(1) \leq x + he_i$ and $U(x+he_i) \leq J(\gamma) + \vare$. Without loss of generality, we may assume that $\gamma(0) \in \partial_{R,M}\Omega$, $\gamma(1)=x+he_i$, and $\gamma_i'(t) > 0$ for $t\in [0,1]$ and $1 \leq i \leq d$. Let $\Phi(z) = \left(z_1,\ldots \frac{x_i}{x_i + h} z_i, \ldots z_d \right)$ and set $\ov{\gamma} = \Phi(\gamma)$. By construction, $\ov{\gamma}$ satisfies $\ov{\gamma}(1) = x$, $\ov{\gamma}_i(t) = \frac{x_i}{x_i + h}\gamma_i(t)$, and $\ov{\gamma}_j(t) = \gamma_j(t)$ for $j\neq i$. As $J(\ov{\gamma}) \leq u(x)$, we have
\begin{equation*}
    U(x+he_i) - U(x) \leq J(\gamma) - J(\ov{\gamma}) + \vare.
\end{equation*}
A simple calculation shows that $\abs{z - \Phi(z)} \leq \abs{\frac{hz_i}{h+x_i}} \leq Ch$ for $z\in [0,2x]^d$. Hence, we have
\begin{equation}
    \abs{\gamma(t) - \ov{\gamma}(t)} \leq \abs{\frac{h\gamma_i(t)}{h+x_i}} \leq \frac{h x_i}{h+x_i}.
\end{equation}
The above gives us
\begin{equation}
    \abs{f(\gamma(t)) - f(\ov{\gamma}(t))} \leq [f]_{C^{0,1}(\ov{\Omega}_{R,M})} \abs{\frac{h\gamma_i(t)}{h+x_i}} \\
    \leq hx_i [f]_{C^{0,1}(\ov{\Omega}_{R,M})} \frac{1}{x_i} \\ 
    \leq h [f]_{C^{0,1}(\ov{\Omega}_{R,M})}.
\end{equation}
We have
\begin{align*}
    J(\gamma) - J(\ov{\gamma}) &= \int_0^1 \left( f(\gamma(t))\left[\gamma_1'(t)\cdots \gamma_d'(t)  \right]^{1/d} - f(\ov{\gamma}(t))\left[\ov{\gamma}_1'(t)\cdots \ov{\gamma}_d'(t)  \right]^{1/d} \right)dt \\
    &= \int_0^1 \left( f(\gamma(t)) - f(\ov{\gamma}(t))\left[\frac{\ov{\gamma}_1'(t)\ldots \ov{\gamma}_d'(t)}{\gamma_1'(t)\cdots \gamma_d'(t)}  \right]^{1/d} \right)\left[\gamma_1'(t)\ldots \gamma_d'(t)  \right]^{1/d} dt \\
    &= \int_0^1  \left( f(\gamma(t)) - \left(\frac{x_i}{x_i + h}\right)^{1/d} f(\ov{\gamma}(t)) \right)\left[\gamma_1'(t)\cdots \gamma_d'(t)  \right]^{1/d} dt.
\end{align*}
Furthermore,
\begin{align*}
 \left( f(\gamma(t)) - \left(\frac{x_i}{x_i + h}\right)^{1/d} f(\ov{\gamma}(t)) \right) &=  \left( f(\gamma(t)) - f(\ov{\gamma}(t)) \right) + \left(1 - \left(\frac{x_i}{x_i + h} \right)^{1/d} \right)f(\ov{\gamma}(t)) \\
 &\leq \left( f(\gamma(t)) - f(\ov{\gamma}(t)) \right) + \left(1 - \frac{x_i}{x_i + h}  \right)f(\ov{\gamma}(t)) \\
 &\leq \left( f(\gamma(t)) - f(\ov{\gamma}(t)) \right) + \frac{h}{x_i} f(\ov{\gamma}(t)) \\
 &\leq h \norm{f}_{C^{0,1}(\ov{\Omega}_{R,M})} + \norm{f}_{C^{0,1}(\ov{\Omega}_{R,M})}\frac{h}{x_i} \\
 &=  h(1 + x_i^{-1}) \norm{f}_{C^{0,1}(\ov{\Omega}_{R,M})} 
\end{align*}
We conclude that
\begin{align*}
    J(\gamma) - J(\ov{\gamma}) &\leq h (1+x_i^{-1}) \norm{f}_{C^{0,1}(\ov{\Omega}_{R,M})}\int_0^1 \left[\gamma_1'(t)\ldots \gamma_d'(t)  \right]^{1/d} \\
    &\leq h (1+x_i^{-1}) \norm{f}_{C^{0,1}(\ov{\Omega}_{R,M})}\prod_{j=1}^d (\gamma_j(1) - \gamma_j(0))^{1/d} \\
    &\leq 2h \norm{f}_{C^{0,1}(\ov{\Omega}_{R,M})}\max_{x\in \ov{\Omega}_{R,M}}  \frac{(x_1\ldots x_d)^{1/d}}{x_i}.
\end{align*}
Observe that the maximum value of $\frac{(x_1\ldots x_d)^{1/d}}{x_i}$ over $\ov{\Omega}_{R,M}$ is attained when $x_j = M$ for $j\neq i$ and $x_i = R^d M^{-(d-1)}$, we have
\begin{align*}
    \max_{x\in \ov{\Omega}_{R,M}} \frac{(x_1\ldots x_d)^{1/d}}{x_i} = R^{-(d-1)} M^{d-1}.
\end{align*}
We conclude that for every $\vare > 0$ and $1 \leq i \leq d$ we have
\begin{equation*}
     U(x+he_i) - U(x) \leq \vare + C h M^{d-1} R^{-(d-1)} \norm{f}_{C^{0,1}(\ov{\Omega}_{R,M})}.
\end{equation*}
and consequently that
\begin{equation*}
     [U]_{C^{0,1}(\ov{\Omega}_{R,M})} \leq C M^{d-1} R^{-(d-1)} \norm{f}_{C^{0,1}(\ov{\Omega}_{R,M})}. \qedhere
\end{equation*}
\end{proof}

\section{Rates of convergence for the longest chain problem}\label{sec:longestchain} 
As discussed in Section \ref{sec:introduction}, nondominated sorting is equivalent to the problem of finding the length of a longest chain in a Poisson point process with respect to the coordinatewise partial order. Given $n \in \N$ and $\rho \in C(\R)$ satisfying \eqref{eq:rhobounds}, let $X_{n\rho}$ denote a Poisson point process on $\R^d$ with intensity $n\rho$. Given a finite set $A \subset \R^d$, let $\ell(A)$ denote the length of the longest chain in the set $A$. Then the Pareto-depth function $U_n$ in $\R^d$ is given by $U_n(x) = \ell([0,x] \cap X_{n\rho})$ where $[0,x] = [0,x_1]\times \ldots \times [0,x_d]$. The scaled Pareto-depth function is defined by $u_n(x) = \frac{d}{c_d} n^{-1/d}U_n(x)$ where $c_d$ is given by \eqref{eq:definitionofcd}. When $S \subset \R^d$ is bounded and Borel measurable, we write $u_n(S)$ to denote $\frac{d}{c_d} n^{-1/d}\ell(S \cap X_{n\rho})$ and $\abs{S}$ to denote its Lebesgue measure. When $\rho$ is constant and $S$ is a simplex of the form $\set{x\in (-\infty,0]^d : 1+x\cdot q \geq 0}$ with $q\in \R^d_+$, one can show that
\begin{align*}
    \lim_{n\to \infty} u_n(S) = d\rho^{1/d} \abs{S}^{1/d} \text{ a.s. }
\end{align*}
In this section, we establish explicit rates of convergence for the length of the longest chain in rectangles and simplices. We begin by stating a simple property of Poisson processes, whose proof is found in \cite{kingman1992poisson}.
\begin{lemma}\label{lem:poissonproperties}
Let $X_\rho$ be a Poisson process on $\R^d$ with intensity function $\rho$, where $\rho\in L^1_{loc}(\R^d)$ is nonnegative. Then given $g_1,g_2\in L^1_{loc}(\R^d)$ with $0\leq g_1\leq \rho \leq g_2$, there exist Poisson point processes $X_{g_1}$ and $X_{g_2}$ such that $X_{g_1} \subseteq X_\rho \subseteq X_{g_2}$.
\end{lemma}
The following result is proved in \cite{bollobas1992height} by Bollob\'as and Brightwell.
\begin{theorem}\label{thm:BollobasbrightwellOrig}
Let $X_{n}$ be a Poisson point process on $[0,1]^d$ with intensity $n$. Then there exists a constant $C_d$ such that for all $n > C_d$ we have
\begin{align*}
\p\left( \abs{U_n([0,1]^d) - \E U_n([0,1]^d) } > C_d t n^{1/2d}\frac{\log n}{ \log \log n}    \right) \leq 4t^2 \exp(-t^2)
\end{align*}
for all $t$ satisfying $2 < t < \frac{ n^{1/2d}}{\log \log n}$. Furthermore, 
 \begin{equation*}
        c_d n^{1/d} \geq \E U_n([0,1]^d) \geq c_d n^{1/d} - C_d n^{1/2d}\frac{\log^{3/2} n}{\log \log n}
    \end{equation*}
    where $c_d$ is given by
    \begin{align*}
        c_d = \lim_{n\to \infty} n^{-1/d}\ell(X_{n}  \cap [0,1]^d) \text{ a.s. }
    \end{align*}
\end{theorem}
Next, we extend Theorem \ref{thm:BollobasbrightwellOrig} to a Poisson process with intensity $n\rho$ where $\rho > 0$ is a constant.
\begin{theorem}\label{thm:Bollobasbrightwell}
Let $X_{n \rho}$ be a Poisson point process on $[0,1]^d$ where $\rho > 0$ is a constant. Then for all $n > C_d\rho^{-1}$ and all $t$ satisfying $2 < t < \frac{(\rho n)^{1/2d}}{\log \log \rho n}$ we have
\begin{align*}
\p\left( \abs{u_n([0,1]^d) - \E u_n([0,1]^d) } > C_d n^{-1/2d}\rho^{1/2d} t\frac{\log \rho n}{ \log \log \rho n}    \right) \leq 4t^2 \exp(-t^2).
\end{align*}
Furthermore,
\begin{equation*}
       d\rho^{1/d} \geq \E u_n([0,1]^d) \geq d\rho^{1/d} - C_d \rho^{1/2d} n^{-1/2d}\frac{\log^{3/2}\rho n}{\log \log \rho n}.
\end{equation*}
\end{theorem}
\begin{proof}
Replace $n$ by $\rho n$ in Theorem \ref{thm:BollobasbrightwellOrig}. Also note that
\begin{equation*}
        \lim_{n\to \infty}  n^{-1/d}\ell(X_{\rho n}  \cap [0,1]^d) = c_d \rho^{1/d} \text{ a.s. } \qedhere
    \end{equation*}
\end{proof}
Next we establish rates of convergence for the longest chain problem in a rectangular box.
\begin{theorem}\label{thm:rectanglerate}
Let $X_{n\rho}$ denote a Poisson point process on $\R^d$ with intensity $n\rho$ where $\rho \in C(\R^d)$ satisfies \eqref{eq:rhobounds}. Given $x,y\in \R^d$ with $x_i < y_i$ for $i=1,\ldots ,d$, let $R = [x,y] := \set{w \in \R^d: x_i \leq w_i \leq y_i \text{ for } i=1,\ldots ,d}$.
\begin{enumerate}
    \item[(a)]
    For all $n > C_d(\sup_R \rho)^{-1} \abs{R}^{-1}$ and $t$ satisfying 
\begin{equation*}
    C_d < t < C_d n^{1/2d} (\sup_R \rho)^{1/2d} \abs{R}^{1/2d} \frac{1}{\log \log n(\sup_R \rho)\abs{R}}
\end{equation*} 
we have
    \begin{align*}
    \p\left(u_n(R) - d(\sup_R \rho)^{1/d}\abs{R}^{1/d} > C_d t n^{-1/2d}(\sup_R \rho)^{1/2d} \abs{R}^{1/2d} \frac{\log^{3/2} (n\abs{R}(\sup_R \rho))}{\log \log (n\abs{R}(\sup_R \rho))} \right) \\
    \leq 4 t^2\exp(-t^2).
    \end{align*}
    
    \item[(b)]
    For all $n > C_d(\inf_R \rho)^{-1} \abs{R}^{-1}$ and $t$ satisfying 
\begin{equation*}
    \log^{1/2}(n(\inf_R \rho)\abs{R}) < t < (\inf_R \rho)^{1/2d} n^{1/2d} \abs{R}^{1/2d} \frac{1}{\log \log n(\inf_R \rho)\abs{R}}
\end{equation*}
we have
    \begin{align*}
    \p\left(u_n(R) - d(\inf_R \rho)^{1/d}\abs{R}^{1/d} < - C_d t n^{-1/2d} (\inf_R \rho)^{1/2d} \abs{R}^{1/2d} \frac{\log (n\abs{R}(\inf_R \rho))}{\log \log (n\abs{R}(\inf_R \rho))}   \right) \\
    \leq 4 t^2\exp(-t^2).
    \end{align*}
\end{enumerate}

\end{theorem}

\begin{proof}
We shall prove only (a), as the proof of (b) is similar. Without loss of generality we may take $R$ to be the rectangle $[0,y]$ with $y \in \R^d_+$. By Lemma \ref{lem:poissonproperties} there exists a Poisson process $\ov{X}_n \supset X_{n\rho}$ on $\R^d$ with intensity function $n\ov{\rho}$ where $\ov{\rho} =  \left(\sup_R \rho\right)\mathbbm{1}_R + \rho \mathbbm{1}_{\R^d \setminus R}$. Given $A \subset \R^d$, let $\ov{u}_n(A) = n^{-1/d}\ell(A\cap \ov{X}_n)$ and set $\Phi(x) = \left(\frac{x_1}{y_1},\ldots \frac{x_d}{y_d}\right)$. Then $Y_n := \Phi(\ov{X}_n)$ is a Poisson process with intensity $n\abs{R}\ov{\rho}$ and $\ov{u}_n(R) = n^{-1/d}\ell([0,1]^d \cap Y_n)$. Let $E$ be the event that
\begin{align*}
\abs{\ov{u}_n(R) - \E \ov{u}_n(R)} \leq C_{d} (\sup_R \rho)^{1/2d} t \abs{R}^{1/2d}n^{-1/2d} \frac{\log n(\sup_R \rho)\abs{R}}{\log \log n(\sup_R \rho)\abs{R}}
\end{align*}
and
\begin{align*}
    0 \geq \E\ov{u}_n(R) - d (\sup_R \rho)^{1/d} \abs{R}^{1/d} \geq - C_{d}(\sup_R \rho)^{1/2d} t \abs{R}^{1/d} n^{-1/2d}  \frac{\log^{3/2}(n(\sup_R \rho)\abs{R})}{\log \log (n(\sup_R \rho)\abs{R})}.
\end{align*}
where $2 < t < \frac{\abs{R}^{1/2d} (\sup_R \rho)^{1/2d} n^{1/2d}}{\log \log (\sup_R \rho)\abs{R} n}$, $n > (\sup_R \rho)^{-1} \abs{R}^{-1}$, and the constant $C_d$ is as in Theorem \ref{thm:Bollobasbrightwell}. By Theorem \ref{thm:Bollobasbrightwell} we have $P(E) \geq 1 - 4t^2 \exp(-t^2)$. Assume that $E$ holds for fixed choices of $t$ and $n$. As $u_n(R) \leq \ov{u}_n(R)$, we have
\begin{align*}
u_n(R) - d(\sup_R \rho)^{1/d}\abs{R}^{1/d} &\leq \ov{u}_n(R) - C_d(\sup_R \rho)^{1/d}\abs{R}^{1/d} \\
&\leq \abs{\ov{u}_n(R) - \E \ov{u}_n(R)} + (\E \ov{u}_n(R) - C_d(\sup_R \rho)^{1/d} \abs{R}^{1/d}) \\
&\leq C_{d} t (\sup_R \rho)^{1/2d} n^{-1/2d} \abs{R}^{1/2d} \frac{ \log^{3/2} n(\sup_R \rho)\abs{R}}{\log \log n(\sup_R \rho)\abs{R}}. \qedhere
\end{align*}
\end{proof}
Now we extend the preceding result to establish rates of convergence for the longest chain in an orthogonal simplex of the form $S_{y,q}$ as in \eqref{eq:gen_simplex}. The lower one-sided rate is easily attained taking the rectangle $R \subset S$ with largest volume and applying Theorem \ref{thm:rectanglerate}. To prove the upper one-sided rate, we embed $S$ into a finite union of rectangles and apply the union bound. The following result verifies the existence of a suitable collection of rectangles.
\begin{lemma}\label{lem:simplexcovering}
Given $y\in \R^d$ and $q \in \R^d_+$, let $S_{y,q}$ be as in \eqref{eq:gen_simplex}. Given $\vare > 0$, there exists a finite collection $\mathcal{R}$ of rectangles covering $S_{y,q}$ satisfying
\begin{align}\label{eq:simplexcovering1}
        C_d\vare \abs{S_{y,q}} \leq \abs{R}^{1/d} \leq \abs{S_{y,q}}^{1/d} + C_d \vare  \abs{S_{y,q}}\text{ for all } R \in \mathcal{R},
    \end{align}
    and
\begin{align}\label{eq:simplexcovering2}
         \dist(z,S_{y,q})\leq C_d \vare \abs{S_{y,q}} \text{ for all } R \in \mathcal{R} \text{ and } z\in R,
    \end{align}
and
\begin{align}\label{eq:simplexcovering3}
        \abs{\mathcal{R}} \leq C_d\vare^{-(d-1)}.
\end{align}
\end{lemma}
\begin{proof}
Without loss of generality, we may take $y = 0$ and prove the statements for the simplex $S_q := S_{0,q}$. Letting $\mathds{1}$ denote the ones vector, we first prove the statements for the simplex $S := \set{x\in [0,\infty)^d: x\cdot \mathds{1}\leq 1}$, and then obtain the general result via reflection and scaling. Let $P = \set{x\in [0,\infty)^d: x\cdot \mathds{1} = 1}$. Fix $x_0 \in P$, and let $\mathcal{R} = \set{[0,x+\vare \mathds{1}]: x\in P \cap (x_0 + \vare \Z^d) }$. It is clear from the definition of $\mathcal{R}$ that \eqref{eq:simplexcovering2} and \eqref{eq:simplexcovering3} hold. By the Arithmetic-Geometric Mean Inequality, for all $x\in P$ we have
\begin{align*}
    \vare \leq \prod_{i=1}^d (x_i + \vare)^{1/d} \leq \frac{1}{d} + \vare
\end{align*}
and it follows that \eqref{eq:simplexcovering1} holds. To show that $S \subset \bigcup_{R\in \mathcal{R}} R$, let $y\in P$. Then there exists $y^* \in (x_0 + \vare \Z^d)$ such that $\abs{y_i - y_i^*} \leq \vare$ for $1 \leq i \leq d$. Hence, $y\in [0, y^*]$, and it follows that $S \subset \bigcup_{R\in \mathcal{R}} R$. This concludes the proof for $S$, and we now leverage this result to prove the statement for the simplex $S_q = \set{x\in (-\infty,0]^d: 1 + x\cdot q \geq 0}$. Let $\Phi(x) = (\frac{-x_1}{q_1},\ldots \frac{-x_d}{q_d})$, so $\Phi(S) = S_q$. Applying the proven result for $S$, there exists a collection of rectangles $\mathcal{R}_1$ covering $S$ and satisfying $\abs{\mathcal{R}_1} \leq C_d \vare^{-(d-1)}$, $\abs{R}^{1/d} \leq \frac{1}{d} + \vare$ for all $R\in \mathcal{R}_1$, and $\dist(z,S_1) \leq C_d \vare$ for $z\in R$ and $R \in \mathcal{R}$. Let $\mathcal{R} = \set{\Phi(R): R \in \mathcal{R}_1}$, and we verify that $\mathcal{R}$ satisfies the required properties. We have $\abs{\mathcal{R}} = \abs{\mathcal{R}_1}$, so \eqref{eq:simplexcovering3} holds. To see that \eqref{eq:simplexcovering1} holds, observe that
\begin{align*}
    \abs{\Phi(R)}^{1/d} = d \abs{R}^{1/d} \abs{S_{q}}^{1/d}
\end{align*}
and
\begin{align*}
    d \vare  \abs{S_q}^{1/d}\leq d \abs{R}^{1/d} \abs{S_q}^{1/d} \leq \abs{S_q}^{1/d} + d \vare \abs{S_q}^{1/d}
\end{align*}
To show \eqref{eq:simplexcovering2}, let $z \in R \in \mathcal{R}_1$. Then there exists $y\in S$ such that $\abs{z - y} \leq C_d \vare$. Hence, we have $\abs{\Phi(z) - \Phi(y)} \leq C_d \abs{S_q}^{1/d} \vare$, and \eqref{eq:simplexcovering2} follows.
\end{proof}
Now we prove our main result of this section.
\begin{theorem}\label{thm:simplexrate}
Let $\rho \in C^{0,1}(\R^d)$ satisfy \eqref{eq:rhobounds}, and given $y \in \R^d$ and $q \in \R^d_+$ let $S_{y,q}$ be given by \eqref{eq:gen_simplex}. Assume that $d\abs{S_{y,q}}^{1/d} \leq 1$. Then for all $k\geq 1$ and $n > C_{d,\rho,k}\abs{S_{y,q}}^{-1} \log(n)^{2d}$ we have
\begin{align*}
    \p\left(u_n(S_{y,q}) - d(\sup_{S_{y,q}} \rho)^{1/d} \abs{S_{y,q}}^{1/d} >  C_{d,k,\rho} n^{-1/2d} \abs{S_{y,q}}^{1/2d} \frac{\log^{2} n}{\log \log n}  \right) \leq C_{d,k} n^{-k}
\end{align*}
and 
 \begin{align*}
    \p\left(u_n(S_{y,q}) - d(\inf_{S_{y,q}} \rho)^{1/d} \abs{S_{y,q}}^{1/d} < - C_{d,k,\rho} n^{-1/2d} \abs{S_{y,q}}^{1/2d}  \frac{\log^{2} n}{\log \log n}   \right) \leq C_{d,k} n^{-k}.
\end{align*}
\end{theorem}
\begin{proof}
We present the proof of the first statement only, as the proof of the second is similar and simpler. Without loss of generality, we may take $y = 0$ and prove the result for the simplex $S_q = S_{0,q}$. We first prove the result for the simplex $S := \set{x\in (-\infty, 0]^d : 1 + x\cdot \mathds{1}\geq 0}$, and then obtain the general result via a scaling argument. Given $\vare > 0$, we may apply Lemma \ref{lem:simplexcovering} to conclude there exists a collection $\mathcal{R}$ of rectangular boxes covering $S$ such that $\abs{\mathcal{R}} \leq C_d\vare^{-(d-1)}$, $C_d \vare \leq \abs{R}^{1/d} \leq \frac{1}{d} + C_d\vare$ for each $R \in \mathcal{R}$, and $\dist(y,S) \leq \vare$ for $y\in R \in \mathcal{R}$. Set $\ov{R} = \bigcup_{R\in \mathcal{R}} R$. As $\rho \in C^{0,1}(\R^d)$, we have $\sup_{\ov{R}} \rho \leq \sup_{S} \rho + [\rho]_{C^{0,1}(\R^d)} \vare$ for $y\in \ov{R}$. Let $K = \left(\sup_{S} \rho + [\rho]_{C^{0,1}(\R^d)} \vare\right)$. By Lemma \ref{lem:poissonproperties} there exists a Poisson process $\ov{X}_n \supset X_{n\rho}$ on $\R^d$ with intensity function $n\ov{\rho}$ where $\ov{\rho} = K\mathbbm{1}_{\ov{R}} + \rho \mathbbm{1}_{\R^d \setminus \ov{R}}$. Given $A \subset \R^d$, let $\ov{u}_n(A) = \frac{d}{c_d}n^{-1/d}\ell(A\cap \ov{X}_n)$. As $\ov{X}_n \supset X_{n\rho}$ and $S \subset \ov{R}$, we have $u_n(S) \leq \ov{u}_n(S) \leq \max_{R \in \mathcal{R}} \ov{u}_n(R)$. Let $E_R$ be the event that
\begin{align}\label{eq:Rbound}
    \ov{u}_n(R) - dK^{1/d}\abs{R}^{1/d} \leq C_d \nu K^{1/2d}
\end{align}
where $\nu = tn^{-1/2d} \frac{\log^{3/2} (n K)}{\log \log (nK)}$. For any $\vare < C_d$, $n > C_d\abs{R}^{-1}K^{-1}$, and $2 < t < \frac{(nK)^{1/2d} }{\log \log nK}$, we have $\p(E_R) \geq 1 - 4t^2\exp(-t^2)$ by Theorem \ref{thm:rectanglerate}. Letting $E$ be the event that \eqref{eq:Rbound} holds for all $R\in \mathcal{R}$, we have $\p(E) \geq 1 - C_d \vare^{-d+1}t^2 \exp(-t^2)$ by the union bound. Given $k \geq 1$, set $t = \sqrt{Ck \log(n)}$ for a constant $C$ chosen large enough so $n^{1/2} n^{-Ck} \leq n^{-k}$, and let $\vare = t n^{-1/2d}$. Observe that the hypotheses of Theorem \ref{thm:rectanglerate} are satisfied when $nK > C_{d,k} \log(n)^{2d}$, so we have $\p(E) \geq 1 - C_{d,k} n^{1/2} n^{-Ck} \geq 1 - C_{d,k} n^{-k}$. If $E$ holds, then using $S \subset \ov{R}$, $\abs{R}^{1/d} \leq \frac{1}{d} + C_d \vare$, and \eqref{eq:Rbound}, we have
\begin{align}
    u_n(S) \leq \ov{u}_n(S) &\leq \max_{R\in \mathcal{R}} \ov{u}_n(R) \\
    &\leq K^{1/d} + C_d \nu (K^{1/2d} + K^{1/d}).
\end{align}
Now we obtain the stated result for the simplex $S_q = \set{x\in (-\infty, 0]^d : 1 + q\cdot x \geq 0}$. Let $\Phi(x) = (q_1 x_1,\ldots q_d x_d)$, and observe that $\Phi(S_q) = S$. Then $Y_n := \Phi(X_{n\rho})$ is a Poisson point process of intensity $n\tilde{\rho}$ where $\tilde{\rho} = \abs{\det \Phi}^{-1} n\rho = d^{d} n\rho\abs{S_q}$. As $\Phi$ preserves the length of chains, we have $\ell(X_{n\rho} \cap S_q) = \ell(Y_n \cap S)$. Let $\vare = n^{-1/2d}\sqrt{Ck \log n}$, $\tilde{K} = \left(\sup_{S} \tilde{\rho} + [\tilde{\rho}]_{C^{0,1}(\R^d)} \vare\right)$, $K = \left(\sup_{S_q} \rho + [\rho]_{C^{0,1}(\R^d)} \vare\right)$, and $\nu =  \frac{\vare \log^{3/2} n \tilde{K}}{\log \log n\tilde{K}}$. Let $E$ be the event that
\begin{align}
    u_n(S_q) &\leq \tilde{K}^{1/d} + C_d \nu (\tilde{K}^{1/2d} + \tilde{K}^{1/d}).
\end{align}
If $E$ holds, then using $\tilde{K} = d^d \abs{S_q} K$ and $d \abs{S_q}^{1/d} \leq 1$, we have
\begin{align*}
    u_n(S_q) &\leq \tilde{K}^{1/d} + C_d \nu (\tilde{K}^{1/2d} + \tilde{K}^{1/d}) \\
    &= d \abs{S_q}^{1/d} K^{1/d} + C_d \nu \abs{S_q}^{1/2d} (1 + K^{1/2d}) \\
     &\leq d \abs{S_q}^{1/d} (\sup_S \rho)^{1/d} + C_{d,\rho} \nu \abs{S_q}^{1/2d} \\
     &\leq d \abs{S_q}^{1/d} (\sup_S \rho)^{1/d} + C_{d,k,\rho} n^{-1/2d} \abs{S_q}^{1/2d} \frac{\log^2 n}{\log \log n}.
\end{align*}
Using our result for $S$ and $\ell(X_{n\rho} \cap S_q) = \ell(Y_n \cap S)$, we have $\p(E) \geq 1 - C_{d,k} n^{-k}$ for $n > C_{d,k} \tilde{K}^{-1} \log(n)^{2d}$.


\end{proof}

\section{Lemmas For Proving Convergence Rates}\label{sec:supporting}
In this section we establish our primary lemma for proving Theorems \ref{thm:mainaux} and \ref{thm:mainfull}. In particular, we prove a maximum principle type result that shows that if $\varphi$ is a semiconcave (see Definition \ref{def:semiconvexity}), strictly increasing supersolution of (\ref{eq:PDEaux}), then the maximum of $u_n - \varphi$ occurs in a neighborhood of $\partial_R \Omega$ with high probability. An analogous result holds for $\varphi - u_n$ when $\varphi$ is a semiconvex, strictly increasing subsolution of (\ref{eq:PDEaux}). In proving Lemma \ref{thm:maxnearboundary} we shall employ Lemmas \ref{lem:chaininAn} and \ref{lem:mindisttoboundary}, whose proofs are presented after Lemma \ref{thm:maxnearboundary}.
\begin{lemma}\label{thm:maxnearboundary}
Let $\rho \in C^{0,1}(\R^d)$ satisfy \eqref{eq:rhobounds}. Given $R \in (0,1)$, let $\varphi \in C(\Omega_R)$ be a non-negative function such that $0<\uv{\gamma} \leq \varphi_{x_i} \leq \ov{\gamma}$ holds in the viscosity sense on $\Omega_R$ for some constants $\ov{\gamma} \geq 1$ and $\uv{\gamma} \leq 1$. Given $\alpha \geq 1$, $k \geq 1$, $\delta \in (0,1)$, and $\vare \in (0,1)$, let $R_n = \ov{\gamma}^{1/2} n^{-1/2d}\vare^{-1/2}\frac{\log^{2}( n)}{\log \log ( n)}$. Then there exist constants $C_d \leq 1$ and $C_{d,k,\rho} \geq 1$ such that when $\vare  \leq C_d\min\set{\alpha^{-1} \uv{\gamma}^{2} \ov{\gamma}^{-1}\log(n)^{-2},  R^{d-1} \delta}$ and $1 \geq\lambda \geq C_{d,k,\rho}(R_n + \uv{\gamma}^{-2}\alpha\vare)$ the following statements hold.
\begin{enumerate}
\item[(a)]
Suppose $\varphi$ is semiconcave on $\Omega_R$ with semiconcavity constant $\alpha \geq 1$ and satisfies $H(D \varphi) \geq \rho$ on $\Omega_{R+\delta}$. Assume further that 
\begin{align*}
    \sup_{\Omega_R} \left(u_n - (1+\lambda) \varphi  \right) \geq 2\vare.
\end{align*}
Then there exists a constant $C_{d,\rho,k} \geq 1$ such that for all $n>1$ with $n^{1/d} > C_{d,\rho, k} \vare^{-1} \ov{\gamma} \log(n)^{2}$ we have
\begin{align*}
\p\left( \sup_{\Omega_R} \left(u_n - (1+\lambda) \varphi  \right) = \sup_{\Omega_{R+\delta}} \left(u_n - (1+\lambda) \varphi  \right) \right) \leq C_{d,k} \vare^{-6d } n^{-k}.
\end{align*}
\item[(b)]
Suppose $\varphi$ is semiconvex on $\Omega_R$ with semiconvexity constant $\alpha \geq 1$ and satisfies $H(D \varphi) \leq \rho$ on $\Omega_{R+\delta}$. Then there exists a constant $C_{d,\rho,k} \geq 1$ such that for all $n>1$ with $n^{1/d} >  C_{d,\rho,k}\alpha^{-1}\uv{\gamma}^{2}\vare^{-2}$ we have
\begin{align*}
\p\left( \sup_{\Omega_R}\left((1-\lambda) \varphi - u_n  \right) = \sup_{\Omega_{R+\delta}} \left((1-\lambda) \varphi - u_n  \right) \right) \leq C_{d,k} \vare^{-6d} n^{-k}.
\end{align*}
\end{enumerate}
\end{lemma}
\begin{proof}
First we introduce the notation used throughout the proof. Let $\Omega_R^\vare = \Omega_R \cap d^{-1/2}\vare^3\Z^d$ and we define a collection of simplices $\s = \set{S_{x,s}: x\in \Omega_{R+\delta}^\vare,s\in \Gamma_\vare}$ where $S_{x,s}$ is given by \eqref{eq:gen_simplex} and
\begin{align*}
    \Gamma_\vare = \set{s \in \vare^3\Z^d : (4\ov{\gamma})^{-1}\vare \leq s_i \leq 4\uv{\gamma}^{-1}\vare}.
\end{align*}
Lemma \ref{lem:mindisttoboundary} implies that $S \subseteq \Omega_R$ for all $S \in \mathcal{S}$ when $\vare \leq C_d \delta R^{d-1}$. Let $E_S$ be the event that
\begin{align}\label{eq:simplexbound}
    u_n(S) - d(\sup_S \rho^{1/d})\abs{S}^{1/d} \leq \frac{C_{d,k,\rho} \abs{S}^{1/2d} \log^{2} n}{n^{1/2d} \log \log n}
\end{align}
and let $E$ be the event that $E_S$ holds for all $S \in \mathcal{S}$. For any $n >C_{d,k,\rho} \abs{S}^{-1} \log(n)^{2d}$, we have $\p(E_S) \geq 1-C_{d,k} n^{-k}$ by Theorem \ref{thm:simplexrate}. By choice of $\Gamma_\vare$, we have $(4\ov{\gamma})^{-1}\vare \leq d\abs{S}^{1/d} \leq 4\uv{\gamma}^{-1} \vare$ for $S \in \mathcal{S}$. As we assume that $n^{1/d} > C_{d,\rho,k}\vare^{-1} \ov{\gamma} \log(n)^{2}$, we have $n > C_{d,k,\rho}\abs{S}^{-1} \log(n)^{2d}$ for all $S \in \mathcal{S}$. As $\abs{\Gamma_\vare} \leq C_d \vare^{-3d}$ and $\abs{\Omega_R^\vare} \leq C_d \vare^{-3d}$, we have $\abs{\mathcal{S}} \leq C_d \vare^{-6d}$. By the union bound, we have $\p\left( E \right) \geq 1-C_{d,k} \vare^{-6d} n^{-k}$. For the remainder of the proof we assume that $E$ holds. Then for each $S \in \mathcal{S}$ we have
\begin{align*}
    u_n(S)  \leq d \abs{S}^{1/d} (\sup_S \rho^{1/d}) + \frac{C_{d,k,\rho} \abs{S}^{1/2d} \log^{2}( n)}{n^{1/2d} \log \log ( n)} \leq d\abs{S}^{1/d}(\sup_S \rho^{1/d}) (1 + C_{d,k,\rho} R_n)
\end{align*}
where $R_n = \ov{\gamma}^{1/2} \vare^{-1/2} n^{-1/2d}\frac{\log^{2}( n)}{\log \log ( n)}$. Let $w = (1+\lambda) \varphi$, and we show that
\begin{align*}
    \sup_{\Omega_R} \left(u_n - w  \right) \neq \sup_{\Omega_{R+\delta}} \left(u_n - w  \right).
\end{align*}
Assume for contradiction that
\begin{align*}
    \sup_{\Omega_R} \left(u_n - w  \right) = \sup_{\Omega_{R+\delta}} \left(u_n - w  \right).
\end{align*}
Since $\varphi$ is semiconcave with semiconcavity constant $\alpha$, at almost every $x\in \Omega_R$ we have that $\varphi$ is twice differentiable at $x$ with $D^2 \varphi(x) \leq \alpha I$. Hence, there is $x_n \in \Omega_{R+\delta}$ such that $D^2 \varphi(x_n) \leq \alpha I$ and $u_n(x_n) - w(x_n) + \vare^3 \geq \sup_{\Omega_R} (u_n - w)$. Therefore, we have
\begin{align}\label{eq:touchingproperty}
u_n(y) - u_n(x_{n}) \leq w(y) - w(x_{n}) + \vare^3 \text{ for all } y\in \Omega_R.
\end{align}
As we assume $\varphi \geq 0$ and $\sup_{\Omega_R}(u_n - w) \geq 2\vare$, we have $u_n(x_n) \geq \vare$. We define the sets
\begin{align*}
A_n &= \set{x\in [0,x_n]\cap \Omega_R: u_n(x_n) - u_n(x)\leq \vare} \\
A &= \set{x\in [0,x_n]\cap \Omega_R: w(x_n) - w(x)\leq \vare + \vare^3}.
\end{align*}
By Lemma \ref{lem:chaininAn} we have $u_n(A_n) \geq \vare$. By (\ref{eq:touchingproperty}) we have $A_n \subset A$. As $w_{x_i} = (1+\lambda)\varphi_{x_i} \geq \varphi_{x_i}$, we have $A \subset B(x_n,2 \uv{\gamma}^{-1}\vare)$. By Taylor expansion, we have
\begin{align*}
w(x) &\leq w(x_n) + D w(x_n)\cdot (x-x_n) + (1+\lambda)\alpha \abs{x-x_n}^2 \\
&\leq w(x_n) + D w(x_n)\cdot (x-x_n) + 2 \alpha \abs{x-x_n}^2.
\end{align*}
Hence, when $x\in A$ we have
\begin{align*}
-\vare-\vare^3 \leq w(x) - w(x_n) &\leq D w(x_n )\cdot (x - x_n ) + 2\alpha \abs{x-x_n }^2 \\
&\leq D w(x_n )\cdot (x - x_n ) + C\alpha \uv{\gamma}^{-2}\vare^2.
\end{align*}
Letting $p_i = \frac{C \alpha \vare^2 \uv{\gamma}^{-2} + \vare}{v_{x_i}(x_n)}$, we have $A \subseteq S_{x_n, p}$. We show there exist $y\in \Omega_R^\vare$ and $q\in \Gamma_\vare$ such that $S_{x_{n},p} \subseteq S_{y,q}$. Let $y\in \Omega^{\vare}_R$ such that $x_{n} \leq y$ and $\abs{y - x_n}\leq \vare^3$. Letting $\mathds{1}$ denote the all ones vector, we have $S_{x_n, p} \subseteq S_{y,p+\vare^3 \mathds{1}}$. We may choose $q$ so $p + 2\vare^3 \mathds{1} \geq q \geq p + \vare^3 \mathds{1}$ provided that $\ov{\gamma}^{-1}\vare \leq p_i \leq 4\uv{\gamma}^{-1}\vare - 2\vare^3$ for each $i$, which holds when $\vare \leq C\alpha^{-1} \uv{\gamma}^{2}$ and $\alpha \geq 1$. Then $S_{y,p+\vare^3 \mathds{1}} \subset S_{y,q}$. Using that $(q_1\ldots q_d)^{1/d} = d \abs{S_{y,q}}^{1/d}$, we have
\begin{align*}
\vare &\leq u_n(A_n) \\
&\leq u_n(S_{y,q}) \\
&\leq (1+C_{d,k,\rho} R_n)(\sup_{S_{y,q}} \rho)^{1/d} (q_1\ldots q_d)^{1/d} \\
&\leq (1 + C_{d,k,\rho} R_n) (\sup_{S_{y,q}} \rho)^{1/d} \left( \prod_{i=1}^d \left\{ \frac{ C \alpha \vare^2 \uv{\gamma}^{-2}  + \vare}{w_{x_i}(x_n)} + 2\vare^3 \right\} \right)^{1/d} \\
&\leq (1+ C_{d,k,\rho} R_n) (\sup_{S_{y,q}} \rho)^{1/d} \left( \prod_{i=1}^d \left\{ \frac{ C \alpha \vare^2 \uv{\gamma}^{-2} + \vare + 2 \ov{\gamma}\vare^3 }{w_{x_i}(x_n )} \right\}\right)^{1/d}.
\end{align*}
As our assumptions on $\vare$ imply $\vare \leq C \ov{\gamma}^{-1}$, we have $\ov{\gamma}\vare^3 \leq C \alpha \vare^2 \uv{\gamma}^{-2}$. Hence we have
\begin{align*}
\left(w_{x_1}(x_n)\cdot \ldots w_{x_d}(x_n)\right)^{1/d} &\leq (1 + C_{d,k,\rho} R_n)(\sup_{S_{y,q}} \rho)^{1/d} \left(1 + C \uv{\gamma}^{-2}\alpha  \vare \right).
\end{align*}
As we assume $\rho \in \Czeroone$ and $\rho \geq \rhomin > 0$, we have $\rho^{1/d} \in C^{0,1}(\Omega_R)$, hence $\sup_{S_{y,q}} \rho^{1/d} \leq \rho(x_n)^{1/d} + C_{d,\rho}\vare \uv{\gamma}^{-1}$. By our assumption $n^{1/d} > C_{d,k,\rho}\vare^{-1} \ov{\gamma} \log(n)^2$, we have $R_n \leq C_{d,k,\rho}$. As $\vare \leq C\alpha^{-1} \uv{\gamma}^2 \ov{\gamma}^{-1}$, we have $\uv{\gamma}^{-2}\alpha \vare \leq C$. Hence we have
\begin{align*}
\left(w_{x_1}(x_n)\cdot \ldots w_{x_d}(x_n)\right)^{1/d} \leq \rho(x_n)^{1/d} + C_{d,k,\rho}(R_n + \uv{\gamma}^{-2}\alpha \vare ).
\end{align*}
Applying Proposition \ref{prop:perturb} with $\lambda \geq C_{d,k,\rho}(R_n + \uv{\gamma}^{-2}\alpha\vare)$, we obtain a contradiction. We conclude that we must have
\begin{equation*}
    \sup_{\Omega_R} \left(u_n - w  \right) \neq \sup_{\Omega_{R+\delta}} \left(u_n - w  \right).
\end{equation*}
We now prove (b). Let $\Omega_R^\vare$, $\Gamma_\vare$, and $\mathcal{S}$ be as in (a), and let $E_S$ be the event that
\begin{align}\label{eq:simplexbound}
    u_n(S) - d(\inf_S \rho^{1/d})\abs{S}^{1/d} \geq -\frac{C_{d,k,\rho} \abs{S}^{1/2d} \log^{2} n}{n^{1/2d} \log \log n}
\end{align}
and let $E$ be the event that $E_S$ holds for all $S \in \mathcal{S}$. For any $n >C_{d,k,\rho} \abs{S}^{-1} \log(n)^{2d}$, we have $\p(E_S) \geq 1-C_{d,k} n^{-k}$ by Theorem \ref{thm:simplexrate}. By choice of $\Gamma_\vare$, we have $\ov{\gamma}^{-1}\vare \leq d\abs{S}^{1/d} \leq 4\uv{\gamma}^{-1} \vare$ for $S \in \mathcal{S}$. As we assume that $n^{1/d} > C_{d,\rho,k}\vare^{-2}\alpha^{-1} \uv{\gamma}^2$ and $\vare \leq \uv{\gamma}^2 \ov{\gamma}^{-1}\log(n)^{-2}$, for all $S \in \mathcal{S}$ we have 
\begin{align*}
    n^{1/d} &> C_{d,\rho,k}\vare^{-2}\alpha^{-1} \uv{\gamma}^2 \geq \vare^{-1}\ov{\gamma}\log(n)^2 \geq C_{d,k,\rho}\abs{S}^{-1/d} \log(n)^{2}.
\end{align*}
By the union bound, we have $\p\left( E \right) \geq 1-C_{d,k} \vare^{-6d} n^{-k}$. For the remainder of the proof we assume that $E$ holds. Then for each $S \in \mathcal{S}$ we have
\begin{align*}
    u_n(S)  \geq d \abs{S}^{1/d} (\inf_S \rho^{1/d}) - \frac{C_{d,k,\rho} \abs{S}^{1/2d} \log^{2}( n)}{n^{1/2d} \log \log ( n)} \geq d\abs{S}^{1/d}(\inf_S \rho^{1/d}) (1 - C_{d,k,\rho} R_n)
\end{align*}
where $R_n = \ov{\gamma}^{1/2} \vare^{-1/2} n^{-1/2d}\frac{\log^{2}( n)}{\log \log ( n)}$. Let $w = (1-\lambda) \varphi$, and we show that
\begin{align*}
    \sup_{\Omega_R} \left(w - u_n  \right) \neq \sup_{\Omega_{R+\delta}} \left(w - u_n  \right).
\end{align*}
Assume for contradiction that
\begin{align*}
    \sup_{\Omega_R} \left(w - u_n  \right) = \sup_{\Omega_{R+\delta}} \left(w- u_n  \right).
\end{align*}
Since $\varphi$ is semiconvex with semiconvexity constant $\alpha$, at almost every $x\in \Omega_R$ we have that $\varphi$ is twice differentiable at $x$ with $D^2 \varphi(x) \geq -\alpha I$. Hence, there is $x_n \in \Omega_{R+\delta}$ such that $D^2 \varphi(x_n) \geq -\alpha I$ and $w(x_n) - u_n(x_n) + \frac{d}{c_d}n^{-1/d} \geq \sup_{\Omega_R} (w - u_n)$. Therefore, we have
\begin{align}\label{eq:touchingproperty}
u_n(x_n) - u_n(x) \leq w(x_n) - w(x) + \frac{d}{c_d}n^{-1/d} \text{ for all } y\in \Omega_R.
\end{align}
We define the sets
\begin{align*}
B_n &= \set{x\in [0,x_n]\cap \Omega_R: u_n(x_n) - u_n(x)\leq \vare - \frac{d}{c_d}n^{-1/d}} \\
B &= \set{x\in [0,x_n]\cap \Omega_R: w(x_n) - w(x)\leq \vare - \frac{2d}{c_d}n^{-1/d}}.
\end{align*}
By (\ref{eq:touchingproperty}) we have $B \subset B_n$. By Lemma \ref{lem:chaininAn} we have $u_n(B_n) \leq \vare$. Letting $p_i = \frac{\vare - K \alpha \vare^2 \uv{\gamma}^{-2}}{v_{x_i}(x_n)}$, we have $S_{x_n, p} \subset B(x_n, \uv{\gamma}^{-1}\vare)$. For $x\in S_{x_n,p}$ we have $1 + (x-x_n)\cdot p^{-1} \geq 0$, hence
\begin{align}\label{eq:sxnpinequalityb}
    Dw(x_n) \cdot (x_n - x) \leq \vare - K \alpha \vare^2 \uv{\gamma}^{-2}.
\end{align}
Using \eqref{eq:sxnpinequalityb}, $\abs{x-x_n}\leq \uv{\gamma}^{-1}\vare$, and $n^{-1/d} \leq \alpha \uv{\gamma}^{-2}\vare^2$, we have
\begin{align*}
w(x_n) - w(x) &\leq Dw(x_n)\cdot (x_n - x) + \alpha \abs{x-x_n}^2 \\
&\leq \vare - K \alpha \uv{\gamma}^{-2} \vare^2 + \alpha \uv{\gamma}^{-2} \vare^2 \\
&\leq \vare - (K-1)n^{-1/d}.
\end{align*}
Choosing $K$ so $(K-1) = \frac{2d}{c_d}$, we have $x\in B$. Hence, $S_{x_n, p} \subset B$. We now show there exist $y\in \Omega_R^\vare$ and $q\in \Gamma_\vare$ such that $S_{y,q} \subseteq S_{x_n,p}$. Let $y\in \Omega^{\vare}_R$ such that $x_{n} \geq y$ and $\abs{y - x_n}\leq \vare^3$. Letting $\mathds{1}$ denote the ones vector, we have $S_{y, p-\vare^3 \mathds{1}} \subseteq S_{x_n,p}$. We may choose $q$ so $p - \vare^3 \mathds{1} \geq q \geq p - 2\vare^3 \mathds{1}$ provided that $\ov{\gamma}^{-1}\vare +2\vare^3 \leq p_i \leq 4\uv{\gamma}^{-1}\vare$ for each $i$, which holds when $\vare \leq \frac{1}{8}\alpha^{-1} \uv{\gamma}^{2} \ov{\gamma}^{-1}$. Then $S_{y,q} \subset S_{x_n, p}$. Using that $(q_1\ldots q_d)^{1/d} = d \abs{S_{y,q}}^{1/d}$, we have
\begin{align*}
\vare &\geq u_n(B_n) \\
&\geq u_n(S_{y,q}) \\
&\geq (1 - C_{d,k,\rho} R_n)(\inf_{S_{y,q}} \rho)^{1/d} (q_1\ldots q_d)^{1/d} \\
&\geq (1 - C_{d,k,\rho} R_n) (\inf_{S_{y,q}} \rho)^{1/d} \left( \prod_{i=1}^d \left\{ \frac{\vare - C_{d} \alpha \vare^2 \uv{\gamma}^{-2} }{w_{x_i}(x_n)} - 2\vare^3 \right\} \right)^{1/d} \\
&\geq (1 - C_{d} R_n) (\inf_{S_{y,q}} \rho)^{1/d} \left( \prod_{i=1}^d \left\{ \frac{\vare - C_{d} \alpha \vare^2 \uv{\gamma}^{-2} - 2 \ov{\gamma}\vare^3 }{w_{x_i}(x_n )} \right\}\right)^{1/d}.
\end{align*}
As our assumptions on $\vare$ imply $\vare \leq \ov{\gamma}^{-1}$, we have $\ov{\gamma}\vare^3 \leq \alpha \vare^2 \uv{\gamma}^{-2}$. Hence we have
\begin{align*}
\left(w_{x_1}(x_n)\cdot \ldots w_{x_d}(x_n)\right)^{1/d} &\geq (1 - C_{d,k,\rho} R_n)(\inf_{S_{y,q}} \rho)^{1/d} \left(1 - C_d \uv{\gamma}^{-2}\alpha  \vare \right).
\end{align*}
As we assume $\rho \in \Czeroone$ and $\rho \geq \rhomin > 0$, we have $(\inf_{S_{y,q}} \rho^{1/d}) \geq \rho(x_n)^{1/d} - C_{d,\rho}\vare \uv{\gamma}^{-1}$. Hence we have
\begin{align*}
    \left(w_{x_1}(x_n)\cdot \ldots w_{x_d}(x_n)\right)^{1/d} &\geq \rho(x_n)^{1/d}(1 - C_{d,k,\rho} R_n)(1 - C_{d,\rho}\vare \uv{\gamma}^{-1}) \left(1 - C_d \uv{\gamma}^{-2}\alpha  \vare \right) \\
    &\geq \rho(x_n)^{1/d} - C_{d,k,\rho}(R_n + \uv{\gamma}^{-2}\alpha \vare ).
\end{align*}
Applying Proposition \ref{prop:perturb} with $\lambda = C_{d,k,\rho}(R_n + \uv{\gamma}^{-2}\alpha\vare)$ we obtain a contradiction. We conclude that we must have
\begin{equation*}
    \sup_{\Omega_R} \left(w - u_n  \right) \neq \sup_{\Omega_{R+\delta}} \left(w - u_n  \right). \qedhere
\end{equation*}

\end{proof}
\begin{lemma}\label{lem:chaininAn}
Given $x_0 \in \Omega_R$ and $\vare > 0$, let 
$$A_n = \set{x\in [0,x_0]\cap \Omega_R : u_n(x_0) - u_n(x) \leq \vare}.$$ 
Then $u_n(A_n) \leq \vare + \frac{d}{c_d}n^{-1/d}$. If $u_n(x_n) \geq \vare$, then additionally we have $u_n(A_n) \geq \vare$.
\end{lemma}
\begin{proof}
First we show that $u_n(A_n) \leq \vare + \frac{d}{c_d}n^{-1/d}$. Let $\C_1$ be a longest chain in $A_n \cap X_{n\rho}$, and let $x$ be the coordinatewise minimal element of $\C_1$. Let $\C_2$ be a longest chain in $[0,x] \cap \Omega_R \cap X_{n\rho}$. Concatenating these chains, we have
\begin{align*}
    u_n(x_0) &\geq u_n(\C_1) + u_n(\C_2) - u_n(\C_1 \cap \C_2) \\
    & = u_n(A_n) + u_n(x) - \frac{d}{c_d}n^{-1/d}.
\end{align*}
Hence,
\begin{align*}
    u_n(A_n) \leq u_n(x_0) - u_n(x) + \frac{d}{c_d}n^{-1/d} \leq \vare + \frac{d}{c_d}n^{-1/d}.
\end{align*}
To prove that $u_n(A_n) \geq \vare$, we must first establish a useful property of longest chains. Given $S \subset \R^d$ and a longest chain $ \set{y_i}_{i=1}^k$ in $S$, we claim that 
\begin{align}\label{eq:longestchainproperty}
    \ell([0,y_j] \cap S \cap X_{n\rho}) = j.
\end{align}
It is clear that $\ell([0,y_j] \cap S \cap X_{n\rho}) \geq j$, as $\set{y_i}_{i=1}^j$ is a chain of length $j$ in $[0,y_j]\cap S \cap X_{n\rho}$. If $\ell([0,y_j] \cap S \cap X_{n\rho}) \geq j + 1$, then there exists a chain $\set{z_i}_{i=1}^{j+1}$ in $[0,y_j] \cap S \cap X_{n\rho}$. Concatenating this chain with $\set{y_i}_{i=j+1}^k$ yields a chain of length $k+1$, contradicting maximality of $\set{y_i}_{i=1}^k$. Now we prove the main result. Let $\set{x_i}_{i=1}^k$ be a longest chain in $[0,x_0] \cap \Omega_R \cap X_{n\rho}$, and let $j = \min \set{i \in \set{1,\ldots ,k} : x_i \in A_n}$. Letting $\C_1 = \set{x_i}_{i=1}^j$ and $\C_2 = \set{x_i}_{i=j}^k$, we have
\begin{align}\label{eq:An}
    u_n(x_0) &= u_n(\C_1) + u_n(\C_2) - u_n(\C_1 \cap \C_2)\\
    &\leq u_n(x_j) + u_n(A_n) - \frac{d}{c_d}n^{-1/d}. 
\end{align}
By \eqref{eq:longestchainproperty} we have $u_n(x_i) = \frac{d}{c_d} n^{-1/d}i$ for $1 \leq i \leq k$.
If $j > 1$, then using $u_n(x_0) - u_n(x_{j-1}) > \vare$ and $u_n(x_0) \geq \vare$ we have
\begin{align}\label{eq:An2}
    u_n(x_0) - u_n(x_j)  \geq \vare - \frac{d}{c_d} n^{-1/d}.
\end{align}
If $j = 1$, then \eqref{eq:An2} is an immediate consequence of $u_n(x_0) \geq \vare$. Combining \eqref{eq:An2} with \eqref{eq:An}, we see that
\begin{align*}
    u_n(A_n) \geq u_n(x_0) - u_n(x) + \frac{d}{c_d} n^{-1/d} \geq \vare.
\end{align*}
\end{proof}

\begin{lemma}\label{lem:mindisttoboundary}
Let $0 < \delta \leq R$. Then given $y \in \Omega_{R+\delta}$, there is a constant $C_d > 0$ such that
\begin{align*}
    \dist(y, \partial_R \Omega) \geq C_d \delta R^{d-1}.
\end{align*}
\end{lemma}
\begin{proof}
Let $x \in \partial_R\Omega$ such that $\abs{x-y} =\dist(y, \partial_R \Omega)$. Then
\begin{align*}
    \set{(1-t)x + ty: t \in (0,1)} \subset \Omega_R \setminus \Omega_{R+\delta}.
\end{align*}
Letting $f(x) = (x_1\ldots x_d)^{1/d}$, we have
\begin{align*}
    f_{x_i}(x) = \frac{(x_1\ldots x_d)^{1/d}}{dx_i}.
\end{align*}
Using that $x_i \geq R^d$ for $x\in \Omega_R$ and that $f(x) \leq R + \delta$ for $x\in \Omega_R \setminus \Omega_{R+\delta}$, we have
\begin{align*}
    \norm{Df}_{L^\infty(\Omega_R \setminus \Omega_{R+\delta})} &\leq \frac{1}{d}R^{-d}(R+\delta) \leq \frac{2}{d} R^{-d+1}.
\end{align*}
Hence, we have
\begin{align*}
    \delta \leq f(y) - f(x) \leq \frac{2}{d} R^{-d+1} \abs{y - x}.
\end{align*}
\end{proof}
Next, we establish estimates on $u_n$, $v_n$, $u$, and $v$ that hold with high probability in a thin tube around $\partial_R \Omega$ and $\partial_0 \Omega$. To do so, we cover the neighborhood with rectangular boxes and apply Theorem \ref{thm:rectanglerate}. In the following Lemma we establish the existence of a suitable collection of rectangles.
\begin{lemma}\label{lem:covering}
The following statements hold.
\begin{enumerate}
    \item[(a)]
    Given $\vare > 0$, there exists a collection $\mathcal{R}$ of rectangles covering $\Omega_0 \setminus \Omega_\vare$ such that $\abs{E}^{1/d} = C\vare$ for each $E \in \mathcal{R}$ and $\abs{\mathcal{R}} \leq C_d \vare^{-d^2}$.
    \item[(b)]
    Given $R \in (0,\frac{1}{2}]$ and $0 < \vare \leq \frac{1}{2}R$, there exists a collection $\mathcal{R}$ of rectangles such that for each $x, y \in \Omega_R \setminus \Omega_{R+\vare}$ with $x < y$ and $[x,y] \subset \Omega_R \setminus \Omega_{R+\vare}$, there exists $E \in \mathcal{R}$ such that $[x,y] \subseteq E$. Furthermore, we have $C_d R^{d-1}\vare \leq \abs{E}^{1/d} \leq C_d R^{-1}\vare$ for each $E \in \mathcal{R}$ and $\abs{\mathcal{R}} \leq C_d R^{-2d(d-1)} \vare^{-2d}$.
\end{enumerate}
\end{lemma}
\begin{proof}
We give the proof of (b) only, as the proof of (a) is similar but simpler. Given $h \in (0,1)$, let 
\begin{align*}
    B &= \set{x\in [0,2]^{d} : R - \vare \leq (x_1\ldots x_d)^{1/d} \leq R + 2 \vare}
\end{align*}
and set $B_h = B \cap h\Z^{d}$. We define
\begin{align*}
    \mathcal{R} = \set{[x_1, x_2]: (x_1, x_2) \in B_h \times B_h \text{ and } x_1 < x_2 \text{ and } [x_1,x_2] \subseteq B}
\end{align*}
and we show that $\mathcal{R}$ has the desired properties when $h$ is appropriately chosen. Let $z, y\in \Omega_R \setminus \Omega_{R+\vare}$ with $z < y$ and $[z,y] \subseteq \Omega_R \setminus \Omega_{R+\vare}$. Let $w(x) = (x_1\ldots x_d)^{1/d}$. Then $w_{x_i} = \frac{(x_1\ldots x_d)^{1/d}}{d x_i}$ and if $x\in B$ and $\vare \leq \frac{1}{2}R$ we have 
\begin{align}\label{eq:wbound}
    C_d R \leq \frac{(x_1\ldots x_d)^{1/d}}{x_i} \leq C_d R^{-(d-1)}.
\end{align}
Hence we have $\abs{z - y} \leq C_d R^{-1} \vare$. Letting $\mathds{1}$ denote the all ones vector, by \eqref{eq:wbound} we have $w(y + h\mathds{1}) \leq (R+\vare) + C_d R^{-(d-1)} h$ and $w(z - h\mathds{1}) \geq R - \vare - C_d R^{-(d-1)} h$. Then there exists a constant $C_d$ such that $[z - 2h\mathds{1}, y + 2h\mathds{1}] \subset B$ when $h = C_d R^{d-1} \vare$. Letting $h = C_d R^{d-1} \vare$, there exist $y^+ \in B_h$ and $y^- \in B_h$ such that $y + 2h\mathds{1}\geq y_i^+ \geq y + h\mathds{1}$ and $z - 2h\mathds{1} \leq y_i^- \leq z - h\mathds{1}$. Letting $A = \abs{[y^-, y^+]}^{1/d}$, we have $A \geq C_d R^{d-1} \vare$ and
\begin{align*}
    A\leq \frac{1}{d}\sum_{i=1}^d (y^+_i - y^-_i) \leq C_d (\abs{z-y} + h) \leq C_d R^{-1}\vare.
\end{align*}
Furthermore, we have
\begin{align*}
    \abs{\mathcal{R}} \leq C_d h^{-2d} \leq C_d R^{-2d(d-1)} \vare^{-2d}.
\end{align*}

\end{proof}
\begin{lemma}\label{lem:unbound}
Let $\rho \in C(\R^d)$ satisfy \eqref{eq:rhobounds}, and let $u_n$ and $v_n$ be given by \eqref{eq:scaledDepth}. and \eqref{eq:scaledDepthFull}. Given $k\geq 1$, $\vare \in (0,1)$, and $n > 0$, let $R_n = n^{-1/2d}\vare^{-1/2}\frac{\log^{2}( n)}{\log \log ( n)}$. Then the following statements hold.
\begin{enumerate}
    \item[(a)]
    For all $n>  C_{d,k}\rhomax^{-1} R^{-d^2}\vare^{-d} \log(n)^{4d}$ we have
    \begin{align*}
    \p\left(\sup_{\Omega_0 \setminus \Omega_{\vare}} v_n > C_{d,k,\rho} \vare \right) \leq C_{d,k} \vare^{-d^2} n^{-k}.
    \end{align*}
    \item[(b)]
    Given $R \in (0,1]$, $\vare \in (0,R]$, and $n>  C_{d,k}\rhomax^{-1} R^{-2d} \vare^{-d} \log(n)^{4d}$ we have
   \begin{align*}
    \p\left(\sup_{\Omega_R \setminus \Omega_{R+\vare}} u_n > C_{d,k,\rho} \vare \right) \leq C_{d,k} \vare^{-3d^2} n^{-k}.
    \end{align*}
\end{enumerate}
\end{lemma}
\begin{proof}
We will prove (b) only, as the proof of (a) is similar. Applying Lemma \ref{lem:covering} (b) with $\alpha = \vare R$, there exists a collection $\mathcal{R}$ of rectangles such that $\sup_{\Omega_R \setminus \Omega_{R+\vare}} u_n \leq \max_{R\in \mathcal{R}} u_n(R)$, $C_d R^{d}\vare\leq \abs{E}^{1/d} \leq C_d \vare$ for $E \in \mathcal{R}$, and $\abs{\mathcal{R}} \leq C_d R^{-2d(d-1)}\alpha^{-2d} \leq C_d \vare^{-3d^2}$. Given $R \in \mathcal{R}$, let $E_R$ be the event that
\begin{align*}
    u_n(R) \leq C_d(\sup_{R} \rho)^{1/d} \vare +  C_d k^{1/2} n^{-1/2d} \vare^{1/2} \frac{\log^{2} n}{\log \log n}. 
\end{align*}
If $E_R$ holds, then our assumption $n>  C_{d,k}\rhomax^{-1} R^{-d^2}\vare^{-d} \log(n)^{4d}$ implies that
$u_n(R) \leq C_{d,k,\rho} \vare$. Furthermore, we have $n > C_{d} (\sup_E \rho)^{-1} \abs{E}^{-1}$ for each $E \in \mathcal{R}$ and $\sqrt{k\log(n)} \leq \frac{(n \abs{E})^{1/2d}}{ \log \log n\abs{E}}$. Applying Theorem \ref{thm:rectanglerate} with $t = \sqrt{2k \log(n)}$,  we have $\p(E_R)  \geq 1 - C_{d} k n^{-k}$ for $n > C_d$. Letting $E$ be the event that $E_R$ holds for all $R \in \mathcal{R}$, we have $\p(E) \geq 1 - C_{d,k} \vare^{-3d^2} n^{-k}$ by the union bound. As $\sup_{\Omega_R \setminus \Omega_{R+\vare}} u_n \leq \max_{R\in \mathcal{R}} u_n(R)$, the result follows.
\end{proof}

\begin{lemma}\label{lem:ubound}
Given $R > 0$ and $\rho \in C(\R^d_+)$ satisfying \eqref{eq:rhobounds}, let $u$ and $v$ denote the solutions of \eqref{eq:PDEaux} and \eqref{eq:PDEfull} respectively. Then the following statements hold.
\begin{enumerate}
    \item[(a)]
    For all $\alpha > 0$ we have
    \begin{equation*}
    \sup_{\Omega_R \setminus \Omega_{R+\alpha}} u \leq d \rhomax^{1/d}\alpha.
\end{equation*}
    \item[(b)]
    We have
    \begin{align*}
        \sup_{\Omega_0} (v - \mathbbm{1}_{\Omega_R} u) \leq d\rhomax^{1/d} R.
    \end{align*}
\end{enumerate}
\end{lemma}
\begin{proof}
(a) Let $w(x) = d\rhomax^{1/d}(x_1\ldots x_d)^{1/d} - d\rhomax^{1/d}R$. Then $w = 0$ on $\partial_R \Omega$ and
\begin{align*}
    (w_{x_1}\ldots w_{x_d})^{1/d} = \rhomax^{1/d} \geq (u_{x_1}\ldots u_{x_d})^{1/d} \text{ on } \Omega_R.
\end{align*}
Observe that $\ov{\Omega_R} = \partial_R \Omega \cup \Omega_R \cup \partial^\ast \Omega_R$, so $\Omega = \Omega_R$ and $\Gamma = \partial_R \Omega$ satisfy the hypotheses of Theorem \ref{thm:maximumprinciple}. As $u$ satisfies \eqref{eq:PDEaux}, $u = 0$ on $\partial_R \Omega$. By Theorem \ref{thm:maximumprinciple}, we have $u \leq w$ on $\Omega_R$. Furthermore, when $x\in \Omega_R \setminus \Omega_{R+\alpha}$, we have $u(x) \leq w(x) \leq d \rhomax^{1/d}\alpha$.

(b) Let $w(x) = d\rhomax^{1/d}(x_1 \ldots x_d)^{1/d}$. Then $w = 0$ on $\partial_0 \Omega$ and $(w_{x_1}\ldots w_{x_d})^{1/d} = \rhomax^{1/d}$. By Theorem \ref{thm:maximumprinciple} we have $v \leq w$ in $\Omega_0$, hence $v \leq d \rhomax^{1/d} R$ in $\Omega_0 \setminus \Omega_R$. Since $u = 0$ on $\partial_R \Omega$, we have $v \leq u + d\rhomax^{1/d} R$ on $\partial_R \Omega$. Furthermore, we have  $(u_{x_1}\ldots u_{x_d})^{1/d} = (v_{x_1}\ldots v_{x_d})^{1/d}$ in $\Omega_R$. Hence, we may apply Theorem \ref{thm:maximumprinciple} to conclude that $v \leq u + d \rhomax^{1/d} R$ within $\Omega_R$, and the result follows.
\end{proof}

\section{Proofs of Convergence Rates}\label{sec:mainrates}

\subsection{Proof of Theorem \ref{thm:mainaux}}\label{subsec:auxrates}
In this section we supply the proof of Theorem \ref{thm:mainaux}. Roughly speaking, the proof approximates $u$ with a semiconcave function $\varphi$, uses Theorem \ref{thm:maxnearboundary} to show that the maximum of $u_n - \varphi$ and $\varphi - u_n$ is likely to be attained in a neighborhood of the boundary, and then applies the boundary estimates established in Lemma \ref{lem:unbound}. To produce a suitable approximation $\varphi$, we use inf and sup convolutions, whose properties are summarized in the following lemma. While the proofs of similar results can be found in standard references on viscosity solutions such as \cite{bardi1997, katzourakis2014introduction, crandall1992}, the estimates are not stated in the sharp form required here. The proofs of the following statements can instead be found in \cite{calder2018lecture}.
\begin{lemma}\label{lem:infconvolution}
Given an open and bound set $\Omega \subset \R^d$, $u \in C^{0,1}(\ov{\Omega})$, and $\alpha > 0$, consider the inf-convolution defined by
\begin{align*}
    u_\alpha(x) &= \inf_{y\in \ov{\Omega}} \left\{u(y) + \frac{1}{2\alpha}\abs{x-y}^2  \right\}.
\end{align*}
Then the following properties hold:
\begin{enumerate}
    \item[(a)]
    $u_\alpha$ is semiconcave with semiconcavity constant $\alpha^{-1}$.
    \item[(b)]
    There exists a constant $C > 0$ such that
    \begin{equation*}
        \norm{u-u_\alpha}_{L^\infty(\ov{\Omega})} \leq C \alpha [u]_{C^{0,1}(\ov{\Omega})}^2.
    \end{equation*}

    \item[(c)]
    There exists a constant $C > 0$ such that
    \begin{equation*}
            [u_\alpha]_{C^{0,1}(\ov{\Omega})} \leq C[u]_{C^{0,1}(\ov{\Omega})}.
    \end{equation*}
    \item[(d)]
    Assume $f\in C^{0,1}(\ov{\Omega})$ and $H \in C^{0,1}_{loc}(\R^d)$. If
    \begin{align*}
        H(Du) \geq f \text{ in } \Omega
    \end{align*}
    then
    \begin{align*}
        H(Du_\alpha) \geq f - C\alpha [u]_{C^{0,1}(\ov{\Omega})} [f]_{C^{0,1}(\ov{\Omega})}  \text{ in } M_\alpha(u)
    \end{align*}
    where 
    \begin{equation*}
        M_\alpha(u) = \set{x\in \Omega: \argmin_{y\in \ov{\Omega}}\left\{u(y) + \frac{1}{2\alpha}\abs{x-y}^2  \right\} \cap \Omega \neq \emptyset}.
    \end{equation*}
    
    \item[(e)]
     Let $y_\alpha \in  \argmin_{y\in \ov{\Omega}}\left\{u(y) + \frac{1}{2\alpha}\abs{x-y}^2  \right\}$. Then there exists a constant $C > 0$ such that
    \begin{align*}
        \abs{x-y_\alpha} \leq C \alpha [u]_{C^{0,1}(\ov{\Omega})}. 
    \end{align*}
\end{enumerate}
\end{lemma}
We are now in a position to tackle the proof of Theorem \ref{thm:mainaux}. We prove the sharper rate in (b) by leveraging the semiconvexity estimates established in Section \ref{sec:semiconvexity}.

\begin{proof}[Proof of Theorem \ref{thm:mainaux}]
(a) Given $\alpha > 0$, let $u_\alpha(x) = \inf_{y\in \Omega_R} \{ u(y) + \frac{1}{2\alpha}\abs{x-y}^2 \}$. By Theorem \ref{thm:lipschitzvar} we have 
\begin{align}\label{eq:lipschitzbound}
    [u]_{C^{0,1}(\ov{\Omega}_R)} \leq C_{d,\rho}R^{-d+1}.
\end{align}
By Lemma \ref{lem:infconvolution}, $u_\alpha$ is semiconcave on $\Omega_R$ with semiconcavity constant $\alpha^{-1}$ and we have
\begin{align}\label{eq:linfinitybound}
    \norm{u - u_\alpha}_{L^\infty(\ov{\Omega}_R)} \leq \alpha [u]_{C^{0,1}(\ov{\Omega}_R)}^2\leq C \alpha R^{-2(d-1)} 
\end{align}
and
\begin{align}\label{eq:infconvsupersolution}
\prod_{i=1}^d (u_\alpha)_{x_i} &\geq \rho - C_{d,\rho} \alpha [u]_{C^{0,1}(\ov{\Omega}_R)} \geq \rho - C_{d,\rho}R^{-d+1}\alpha \text{  in  } M_\alpha(u).
\end{align}
Let $x\in \Omega_R$ such that $\dist(x,\partial_R \Omega) > C\alpha [u]_{C^{0,1}(\ov{\Omega}_R)}$ where $C$ is the constant given in Lemma 6.1 (e), and we show that $x \in M_\alpha(u)$. Given $y_\alpha \in  \argmin_{y\in \ov{\Omega_R}}\left\{u(y) + \frac{1}{2\alpha}\abs{x-y}^2  \right\}$, we have
\begin{align*}
    u(y_\alpha) + \frac{1}{2\alpha}\abs{x - y_\alpha}^2 \leq u(x),
\end{align*}
hence $y_\alpha \leq x$. By Lemma \ref{lem:infconvolution} (e) we have $\abs{x - y_\alpha} \leq C\alpha [u]_{C^{0,1}(\ov{\Omega}_R)} < \dist(x, \partial_R\Omega)$, hence $x \in M_\alpha(u)$. By Lemma \ref{lem:mindisttoboundary}, there exist constants $C_d > 0$ and $C_{d,\rho} > 0$ such that $\Omega_{R+\delta} \subset M_\alpha(u)$ when $\delta = C_{d,\rho} \alpha R^{-2d+2}$. By Lemma \ref{lem:infconvolution} and Theorem \ref{thm:lipschitzvar} we have
\begin{align*}
    [u_\alpha]_{C^{0,1}(\ov{\Omega}_R)} \leq [u]_{C^{0,1}(\ov{\Omega}_R)} \leq C_{d,\rho}R^{-d+1}.
\end{align*}
By \eqref{eq:infconvsupersolution}, we have
\begin{align*}
    C_{d,\rho} R^{(d-1)^2} &\leq (u_\alpha)_{x_i} \leq C_{d,\rho} R^{-(d-1)}
\end{align*}
in the viscosity sense. Given $\lambda > 0$, let $A_1$ denote the event that
\begin{align}\label{eq:maxnearboundary}
    \sup_{\Omega_R} \left(u_n - (1+\lambda) u_\alpha  \right) \neq \sup_{\Omega_{R+\delta}} \left(u_n - (1+\lambda) u_\alpha  \right).
\end{align}
Let $\uv{\gamma} = C_{d,\rho} R^{(d-1)^2}$, $\ov{\gamma} = C_{d,\rho} R^{-(d-1)}$, $\delta = C_{d,\rho} \alpha R^{-2d+2}$, and $\vare \in (0,1)$. Set $\nu = \frac{\log^2 n}{\log \log n}$, $R_n = n^{-1/2d}\ov{\gamma}^{1/2}\vare^{-1/2}\nu$, and $\lambda = C_{d,k,\rho}(R_n + \uv{\gamma}^{-2}\alpha^{-1}\vare + R^{-d+1}\alpha)$. Assume for now that $\sup_{\Omega_R} (u_n - (1+\lambda)u_\alpha) \geq 2\vare$. Then for all $n>1$ with $n^{1/d} > C_{d,\rho,k}R^{-d + 1}\vare^{-1} \log(n)^{2}$ and $\vare$ satisfying
\begin{align}\label{eq:epsilonbounds}
    \vare \leq C_d\min(\alpha\uv{\gamma}^{2}\ov{\gamma}^{-1}\log(n)^{-2}, R^{d-1}\delta)
\end{align}
we have $\p(A_1) \geq 1 - C_{d,k} \vare^{-6d} n^{-k}$ by Lemma \ref{thm:maxnearboundary}. Let $A_2$ be the event that $\sup_{\Omega_R \setminus \Omega_{R+\delta}} u_n \leq C_{d,k,\rho}\delta$. By Lemma \ref{lem:unbound} we have $\p(A_2) \geq 1 - C_{d,k} \delta^{-3d^2} n^{-k}$ for some constant $C_d$. As $\vare \leq \delta$, we have $\p(A_1 \cap A_2) \geq 1 - C_{d,k} \vare^{-C_d} n^{-k}$. If $A_1 \cap A_2$ holds, then using \eqref{eq:maxnearboundary} and \eqref{eq:linfinitybound} we have
\begin{align*}
\sup_{\Omega_R} \left( u_n - u \right) &\leq \sup_{\Omega_R \setminus \Omega_{R+\delta}} \left(u_n - (1+\lambda)u_\alpha \right) + \norm{u}_{L^\infty(\ov{\Omega}_R)}\lambda + \norm{u - u_\alpha}_{L^\infty(\ov{\Omega}_R)} \\
&\leq \sup_{\Omega_R \setminus \Omega_{R+\delta}} u_n + C_{d,\rho}\lambda + C_d \alpha R^{-2d + 2} \\
&\leq C_{d,k,\rho}\delta + C_{d,\rho}\lambda + C_{d,\rho}R^{-2d+2}\alpha
\end{align*}
Using $\uv{\gamma}^{-2} = C_{d,\rho} R^{-2d^2 + 4d - 2}$, we have
\begin{align}\label{eq:errorbound}
    \sup_{\Omega_R} \left( u_n - u \right) \leq C_{d,k,\rho}\left(R_n + R^{-2d+2}\alpha + R^{-2d^2 + 4d - 2}\alpha^{-1}\vare \right).
\end{align}
In the case where $\sup_{\Omega_R} (u_n - (1+\lambda)u_\alpha) < 2\vare$, we cannot apply Lemma \ref{thm:maxnearboundary}, but obtain \eqref{eq:errorbound} by observing that
\begin{align*}
    \sup_{\Omega_R} \left( u_n - u \right) &\leq \sup_{\Omega_R} \left(u_n - (1+\lambda)u_\alpha \right) + \norm{u}_{L^\infty(\ov{\Omega}_R)}\lambda + \norm{u - u_\alpha}_{L^\infty(\ov{\Omega}_R)} \\
    &\leq 2\vare + C_{d,\rho}\lambda + C_{d,\rho}(R^{-2(d-1)}\alpha) \\
    &\leq C_{d,k,\rho}\left(R_n + R^{-2(d-1)}\alpha + R^{-2d^2 + 4d - 2}\alpha^{-1}\vare \right).
\end{align*}
Let $E$ denote the right-hand side of \eqref{eq:errorbound}, and we select the parameters $\alpha$ and $\vare$ to minimize $E$ subject to the constraints from Lemmas \ref{lem:unbound} and \ref{thm:maxnearboundary}. Let  $\vare = \nu n^{-1/2d}R^{\frac{2d^2 - 3d + 1}{2}}$ and $\alpha = R^{\frac{(-2d^2 + 6d - 4)}{2}}\sqrt{\vare} = R^{\frac{-2d^2 + 9d - 7}{4}} n^{-1/4d} \nu^{1/2}$. Then we have
\begin{equation*}
    E \leq C_{d,k,\rho} R^{\frac{(-2d^2 + d + 1)}{4}} \nu^{1/2} n^{-1/4d}.
\end{equation*}
This is subject to the constraints
\begin{align}
    \alpha &\leq R\text{ (from Lemma \ref{lem:unbound})} \label{eq:constraint1}\\
    C_{d,\rho, k}\log(n)^{2}R^{-2} \delta^{-1}   &\leq n^{1/d}  \text{ (from Lemma \ref{thm:maxnearboundary}) } \label{eq:constraint2}\\
       C_{d,\rho, k}\log(n)^{2}R^{-d+1} \vare^{-1}   &\leq n^{1/d}  \text{ (from Lemma \ref{thm:maxnearboundary}) } \label{eq:constraint3}\\
    \vare &\leq C_d R^{2d^2 - 3d + 1} \alpha \log(n)^{-2} \text{ (from Lemma \ref{thm:maxnearboundary}) }\label{eq:constraint4} \\
    \vare &\leq C_d R^{d-1} \delta \text{ (from Lemma \ref{thm:maxnearboundary}) } \label{eq:constraint5}.
\end{align}
As $\vare n^{1/d} = n^{1/2d} R^{\frac{2d^2 - 3d + 1}{2}} \nu$, \eqref{eq:constraint3} is satisfied when $n^{1/d} \geq C_{d,\rho,k} R^{-(2d^2 - d - 1)}\log(n)^C$. We also have \eqref{eq:constraint2}, since $\vare < \delta$. As $R^{2d^2 - 3d + 1} \alpha = R^{\frac{6d^2 - 3d - 3}{4}} n^{-1/4d} \nu^{1/2}$ and $\vare = R^{\frac{2d^2 - 3d + 1}{2}} n^{-1/2d} \nu$, \eqref{eq:constraint4} is satisfied when $n^{1/d} \geq C_{d,k,\rho}R^{-2d^2 + 9d + 1}\log(n)^C$. It also follows that \eqref{eq:constraint5} holds, as $R^{d-1}\delta = R^{-d+1}\alpha \geq R^{2d^2-3d+1}\alpha$. Since $\alpha = R^{\frac{-2d^2 + 9d - 7}{4}} \nu^{1/2} n^{-1/4d}$, \eqref{eq:constraint1} is satisfied when $n^{1/d}\geq C_{d,k,\rho}R^{-(2d^2 - 9d + 11)}\log(n)^C$. It is straightforward to check that the most restrictive of these conditions is $n^{1/d} \geq C_{d,k,\rho} R^{-(2d^2 - d - 1)}\log(n)^C$, hence all constraints are satisfied when this holds. We conclude that
\begin{align*}
\p \left( \sup_{\Omega_R} \left(u_n - u \right) >  C_{d,\rho, k} R^{-\frac{(2d^2 + d + 1)}{4}} n^{-1/4d} \left(\frac{\log^{2} n}{\log \log n}  \right)^{1/2} \right) &\leq C_{d,\rho, k} R^{C_d } n^{-k}.
\end{align*}
(b) By Theorem \ref{thm:mainsemiconvexity}, there exist constants $C_\rho$ and $C_{d,\rho}$ such that $u$ is semiconvex on $\Omega_R$ with semiconvexity constant $C_{d,\rho} R^{-2d+1}$ when $R < C_\rho$. Given $\vare > 0$ and $\lambda > 0$, let $E$ be the event that
\begin{align*}
\sup_{\Omega_R} \left((1-\lambda) u - u_n  \right) \neq \sup_{\Omega_{R + \delta}} \left((1-\lambda)u - u_n  \right).
\end{align*}
Let $\uv{\gamma} = C_{d,\rho} R^{(d-1)^2}$, $\ov{\gamma} = C_{d,\rho} R^{-(d-1)}$, $\alpha = C_{d,\rho}R^{-2d+1}$ and $\vare \in (0,1)$. Set $\nu = \frac{\log^2 n}{\log \log n}$, $R_n = n^{-1/2d}\ov{\gamma}^{1/2}\vare^{-1/2}\nu$, and $\lambda = C_{d,k,\rho}(R_n + \uv{\gamma}^{-2}R^{-2d+1}\vare)$. Then for all $n^{1/d} > C_{d,\rho,k} R^{2d^2 - 2d + 2}\vare^{-2}$ and $\vare$ satisfying \eqref{eq:epsilonbounds}, we have $\p(E) \geq 1 - C_{d,k} \vare^{-6d} n^{-k}$ by Lemma \ref{thm:maxnearboundary}. We assume for the remainder of the proof that $E$ holds. By Lemma \ref{lem:ubound} we have
\begin{align*}
    \sup_{\Omega_R} \left( (1-\lambda)u - u_n \right) = \sup_{\Omega_R \setminus \Omega_{R+\delta}} 
    \left( (1-\lambda)u - u_n \right) \leq C_{d,\rho,k}\delta.
\end{align*}
Letting $\delta = \lambda$ and using $\uv{\gamma}^{-2} R^{-2d+1} = C_{d,\rho} R^{-2d^2 + 2d - 1}$, we have
\begin{align*}
\sup_{\Omega_R} \left( u - u_n \right) &= \sup_{\Omega_R} \left((1-\lambda)u - u_n + \lambda u \right) \\
&\leq \sup_{\Omega_R\setminus \Omega_{R+\delta}} \left((1-\lambda)u - u_n \right) + C_{d,\rho} \lambda \\
&\leq C_{d,\rho,k}(R_n + R^{-2d^2 + 2d - 1} \vare).
\end{align*}
Letting $E = (R_n + R^{-2d^2 + 2d - 1} \vare)$, we select $\vare$ to minimize $E$. Let $\vare =  R^{\frac{4d^2 - 5d + 3}{3}} \frac{\nu^{2/3}}{n^{1/3d}}$, so $R_n = R^{-2d^2 + 2d - 1} \vare$. Then
\begin{align*}
    E = R^{-2d^2 + 2d - 1} \vare = R^{\frac{-2d^2 + d}{3}}n^{-1/3d} \nu^{2/3}.
\end{align*}
From Lemma \ref{thm:maxnearboundary} we have the following constraints
\begin{align}
        C_{d,\rho,k} R^{2d^2 - 2d + 1}  &\leq \vare^2 n^{1/d} = n^{1/3d} R^{\frac{8d^2 - 10d + 6}{3}} \nu^{4/3} \label{eq:bconstraint1}\\
    \vare &\leq C_d\alpha R^{2d^2 - 3d + 1} \log(n)^{-2}= R^{2d^2 - 5d + 1}\log(n)^{-2}\label{eq:bconstraint2}\\
    \vare &\leq C_d R^{d-1}\delta. \label{eq:bconstraint3}
\end{align}
As $\delta = \lambda \geq R^{-2d^2 + 2d - 1}\vare$, it is clear that \eqref{eq:bconstraint3} is satisfied. Observe that \eqref{eq:bconstraint1} is satisfied when $n^{1/d} \geq R^{-2d^2 + 4d - 4}\log(n)^C$. Furthermore, \eqref{eq:bconstraint2} is satisfied when $n^{1/d} \geq C_{d,k,\rho} R^{\frac{-2d^2 + 10d + 2}{3}}\log(n)^C$ for some constant $C$. As the most restrictive of these conditions is $n^{1/d} \geq R^{-2d^2 + 4d - 4}\log(n)^C$, when this is satisfied we can conclude that
\begin{align*}
    \p \left(\sup_{\Omega_R} \left( u - u_n \right) > C_{d,\rho, k} R^{\frac{-2d^2 + d}{3}}{n^{1/3d}} \left(\frac{\log^2 n}{\log\log n} \right)^{2/3}  \right) &\leq C_{d,\rho,k} R^{C_d} n^{-k}. \qedhere
\end{align*}
\end{proof}

\subsection{Proof of Theorem \ref{thm:mainfull}}\label{subsec:fullrates}
We now establish our convergence rate result on $\Omega_0 = (0,1]^d$ by using the auxiliary problem (\ref{eq:PDEaux}) as an approximation to (\ref{eq:PDEfull}). The proof is conceptually straightforward, using $u_n$ as an approximation to $v_n$ and $u$ as an approximation of $v$. As with Theorem \ref{thm:mainaux} we obtain a sharper result in (b) thanks to Theorem \ref{thm:mainsemiconvexity}.
\begin{proof}[Proof of Theorem \ref{thm:mainfull}]
(a) Given $R > 0$, let $E_1$ be the event that
\begin{align}\label{eq:unbound}
    \sup_{\Omega_R} \left(u_n - u \right) \leq  C_{d,\rho, k} R^{-\frac{(2d^2 + d + 1)}{4}} n^{-1/4d} \left(\frac{\log^{2} n}{\log \log n}  \right)^{1/2}.
\end{align}
By Theorem \ref{thm:mainaux} (a) we have $\p(E_1) \geq 1 - C_{d,k,\rho}R^{C_d} n^{-k}$ for all $n^{1/d} \geq C_{d,k,\rho} R^{-(2d^2 - d - 1)}\log(n)^C$. Given any $x\in \Omega_0$ and longest chain $\C$ in $[0,x] \cap X_{n\rho}$, let $\C_1 = \set{y\in \C: y\in \Omega_R}$ and $\C_2 = \set{y\in \C: y\notin \Omega_R}$. Then
\begin{align}\label{eq:vnunbound}
    v_n(x) - u_n(x) \leq v_n(\C) - u_n(\C_1) = v_n(\C_2) \leq \sup_{\Omega_0 \setminus \Omega_R} v_n.
\end{align}
Let $E_2$ denote the event that $\sup_{\Omega_0 \setminus \Omega_R} v_n \leq C_{d,\rho} R$ holds. By Lemma \ref{lem:unbound} we have $\p(E_2) \geq 1 - C_{d,k,\rho} R^{-C_d}  n^{-k}$, for all $n$ with $n^{1/d} \geq C_{d,k,\rho} R^{-d-1} \log(n)^{4}$. Then $\p(E_1 \cap E_2) \geq 1 - C_{d,k,\rho}R^{C_d}n^{-k}$, and for the remainder of the proof, we assume that $E_1 \cap E_2$ holds. Let $\nu = \left(\frac{\log^2 n}{  \log \log n} \right)^{1/2}$. Using \eqref{eq:vnunbound} ,\eqref{eq:unbound}, and $u \leq v$ on $\Omega_R$, we have for all $x\in \Omega_R$ that
\begin{align*}
    v_n(x) - v(x) &\leq (v_n(x) - u_n(x)) + (u_n(x) - u(x)) + (u(x) - v(x)) \\
    &\leq C_{d,k,\rho} (R + R^{-\frac{(2d^2 + d + 1)}{4}} n^{-1/4d} \nu).
\end{align*}
Hence we have
\begin{align*}
    \sup_{\Omega_0} (v_n - v) \leq C_{d,k,\rho} (R + R^{-\frac{(2d^2 + d + 1)}{4}} n^{-1/4d} \nu).
\end{align*}
Letting $R = K n^{-\beta}$, we select the maximum value of $\beta$ such that $R \geq R^{-\frac{(2d^2 + d + 1)}{4}} n^{-1/4d} \nu$ and $n^{1/d} \geq C_{d,k,\rho} R^{-(2d^2 - d - 1)}\log(n)^C$ hold when $n > C_{d,k,\rho}$. These are satisfied when $d\beta (2d^2 + d + 5) < 1$ and $d\beta (2d^2 - d - 1) < 1$, respectively. Letting $\beta = \frac{1}{2d^3 + d^2 + 5d + 1}$,  we have
\begin{align*}
    \sup_{\Omega_0} (v_n - v) \leq C_{d,k,\rho}K R.
\end{align*}
Choosing $K$ so $C_{d,k,\rho}K =1$, we conclude that for all $n > C_{d,\rho,k}$ we have
\begin{equation*}
\p \left( \sup_{\Omega_0} \left(v_n - v \right) >  n^{-1/(2d^3 + d^2 + 5d + 1)} \right) \leq C_{d,\rho,k} n^{-k + C_d}.
\end{equation*}
As this holds for all $k\geq 1$, the result follows.

(b)
Given $R > 0$, let $E$ be the event that
\begin{align}\label{eq:unboundb}
    \sup_{\Omega_R} \left(u - u_n \right) \leq  C_{d,\rho, k} R^{\frac{-2d^2 + d}{3}} n^{-1/3d} \left(\frac{\log^{2} n}{\log \log n}  \right)^{2/3}.
\end{align}
By Theorem \ref{thm:mainaux} (a) we have $\p(E) \geq 1 - C_{d,k,\rho}R^{C_d} n^{-k}$ for all $n$ with \begin{align}\label{eq:ninequalityb}
    n^{1/d} \geq C_{d,k,\rho} R^{-2d^2 + 4d - 4}\log(n)^C.
\end{align}
For the remainder of the proof, we assume that $E$ holds. By Lemma \ref{lem:ubound}, we have $\sup_{\Omega_0 \setminus \Omega_R} (v - \mathds{1}_{\Omega_R}u) \leq C_{d,\rho} R$. Let $\nu = \left(\frac{\log^2 n}{  \log \log n} \right)^{2/3}$. Using \eqref{eq:unboundb} and $u_n \leq v_n$, we have for $x\in \Omega_R$ that
\begin{align*}
    v(x) - v_n(x) &\leq (v(x) - u(x)) + (u(x) - u_n(x)) + (u_n(x) - v_n(x)) \\
    &\leq C_{d,k,\rho} (R +  R^{\frac{-2d^2 + d}{3}} n^{-1/3d} \nu).
\end{align*}
If $x\in \Omega_0 \setminus \Omega_R$, then we have $v(x) - v_n(x) \leq \sup_{\Omega_0\setminus \Omega_R} v \leq C_{d,\rho}R$.
Hence, we have
\begin{align*}
    \sup_{\Omega_0} (v - v_n) \leq C_{d,k,\rho} (R +  R^{\frac{-2d^2 + d}{3}} n^{-1/3d} \nu).
\end{align*}
Letting $R = K n^{-\beta}$, we select $\beta$ to satisfy $R \geq  R^{\frac{-2d^2 + d - 1}{3}} n^{-1/3d} \nu$ and \eqref{eq:ninequalityb} when $n \geq C_{d,\rho,k}$. These are satisfied when $d \beta (2d^2 - d + 3) < 1$ and $d \beta (2d^2 - 4d + 4) < 1$, respectively. Letting $\beta = \frac{1}{2d^3 - d^2 + 3d + 1}$, we have
\begin{align*}
    \sup_{\Omega_R} (v_n - v) \leq C_{d,k,\rho} K n^{-\beta}.
\end{align*}
Choosing $K$ so $C_{d,k,\rho} K = 1$, we conclude that
\begin{equation*}
\sup_{\Omega_0} \left(v_n - v \right) \leq n^{-1/(2d^3 - d^2 + 3d + 1)}  \qedhere
\end{equation*}
\end{proof}
\begin{remark}
Observe that $2d^3 + d^2 + 5d + 1 \geq 2d^3 - d^2 + 3d + 1$ for $d\geq 2$. This shows that making use of the semiconvexity estimates in Theorem \ref{thm:mainsemiconvexity} has genuinely improved the convergence rates.
\end{remark}
The following lemma allows us to extend our results to data modeled by a sequence of \textit{i.i.d.}~random variables instead of a Poisson point process.
\begin{lemma}\label{lem:iidvariablebound}
Let $\set{Y_k}_{k=1}^\infty$ be \textit{i.i.d.}~random variables on $\R^d$ with continuous density $\rho$ and set $\ov{Y}_{n} = \set{Y_1,\ldots ,Y_n}$. Let $X_{n\rho}$ be a Poisson point process with intensity $n\rho$. Let $F: \mathcal{F} \to \R$ where $\mathcal{F} = \set{S \subset \R^d : S \text{ is finite }}$ and suppose that $\p\left(F(X_{n\rho}) > c   \right) \leq K$. Then
\begin{align*}
        \p\left(F(\ov{Y}_{n}) > c  \right) \leq Ke\sqrt{n}.
\end{align*}
\end{lemma}
\begin{proof}
Let $N \sim \text{Poisson}(n)$. Then $\ov{Y}_{N}$ is a Poisson process with intensity $n\rho$, as proven in \cite{kingman1992poisson}. Hence we have
\begin{align*}
    \p\left(F(\ov{Y}_{N}) > c  \right) = \sum_{k=0}^\infty \p\left(F(\ov{Y}_{k}) > c  \right) \frac{k^k e^{-k}}{k!} \leq K.
\end{align*}
By Stirling's Formula, $\frac{n!}{n^n e^{-n}} \leq e\sqrt{n}$. We conclude that $\p\left(F(\ov{Y}_n) > c  \right) \leq Ke\sqrt{n}$.
\end{proof}

\section{Proof of Theorem \ref{thm:mainsemiconvexity}}\label{sec:semiconvexity}
First we define the notion of semiconvexity.
\begin{definition}\label{def:semiconvexity}
A function $u$ is said to be semiconvex with semiconvexity constant $C$ on a domain $\Omega$ if for all $x\in \Omega$ and $h\in \R^d$ such that $x\pm h \in \Omega$ we have
\begin{align*}
u(x+h) - 2u(x) + u(x-h) \geq -C\abs{h}^2.
\end{align*}
A function $u$ is said to be semiconcave with semiconcavity constant $C$ if $-u$ is semiconvex with semiconvexity constant $C$.
\end{definition}
We begin by showing that estimates on the semiconvexity constant of $u$ near the boundary automatically extend to the whole domain provided $\rho^{1/d}$ is semiconvex. The key ingredient in the proof is the concavity of the Hamiltonian $L(p) = (p_1\ldots p_d)^{1/d}$.
\begin{theorem}\label{thm:semiconvexboundsplitdomain}
 Let $\Omega \subset [0,M]^d$ be an open set, and assume that $\partial \Omega = \Gamma \cup \partial^\ast \Omega$ where $\Gamma \subset \partial \Omega$ is closed and $\partial^\ast \Omega$ is as in \eqref{eq:partialstaromega}. Given $\vare > 0$, let $\Omega_\vare = \set{x \in \R^d: \dist(x,\Omega) < \vare}$ and suppose that $u \in C(\Omega_\vare)$ satisfies $(u_{x_1}\ldots u_{x_d})^{1/d} = \rho^{1/d}$ in the viscosity sense on $\Omega_\vare$, where $\rho \in C(\Omega_\vare)$ satisfies \eqref{eq:rhobounds} and $\rho^{1/d}$ is semiconvex on $\Omega$ with semiconvexity constant $K_\rho$. Suppose there exists $h\in \R^d$ such that $\abs{h} < \vare$ and
 \begin{align}\label{eq:semiconvexU}
    u(x+h) - 2u(x) + u(x-h) \geq -K_u\abs{h}^2 \text{ for } x\in \Gamma.
    \end{align}
Then we have
 \begin{equation*}
     u(x+h) - 2u(x) + u(x-h) \geq - (1+\norm{u}_{L^\infty(\Omega_\vare)})\max (K_u,  \rhomin^{-1/d}K_\rho)\abs{h}^2 \text{ for } x\in \Omega.
 \end{equation*}
\end{theorem}
\begin{proof}
Set $L(p) = (p_1\ldots p_d)^{1/d}$ and $w(x) = \frac{u(x+h) + u(x-h)}{2}$, and we show that $w$ satisfies $L(Dw) \geq \rho^{1/d} - K_\rho \abs{h}^2$ on $\Omega$. Given $x_0 \in \Omega$, let $\varphi \in C^\infty(\R^d)$ such that $w - \varphi$ has a local minimum at $x_0$. Without loss of generality we may assume that $w(x_0) = \varphi(x_0)$ and $w - \varphi$ has a strict global minimum at $x_0$. Let
\begin{align*}
\Phi(x,y) = \frac{1}{2}u(x) + \frac{1}{2}u(y) - \varphi\left(\frac{x+y}{2}  \right) + \frac{\alpha}{2} \abs{x - y - 2h}^2.
\end{align*}
Since $\Phi$ is continuous, it attains its minimum at some $(x_\alpha, y_\alpha) \in \ov{\Omega} \times \ov{\Omega}$. Furthermore, we have 
\begin{align}\label{eq:phileqzero}
    \Phi(x_\alpha,y_\alpha) \leq \Phi(x_0 + h, x_0 - h) = w(x_0) - \varphi(x_0) = 0.
\end{align}
As $u$ and $\varphi$ are bounded on $\ov{\Omega}$, we have
\begin{align}\label{eq:xalphabound}
    \frac{\alpha}{2}\abs{x_\alpha - y_\alpha - 2h}^2 \leq C_d.
\end{align}
By compactness of $\ov{\Omega} \times \ov{\Omega}$, there exists a sequence $\alpha_n \to \infty$ such that $\set{x_{\alpha_n}}$ and $ \set{y_{\alpha_n}}$ are convergent sequences. Set $x_n = x_{\alpha_n}$, $y_n = y_{\alpha_n}$, and let $(\widehat{x}, \widehat{y}) = \lim_{n\to \infty} (x_n,y_n)$. 
By \eqref{eq:xalphabound} we have $\widehat{x} - \widehat{y} = 2h$, and now we verify that $(\widehat{x}, \widehat{y}) = (x_0 + h, x_0 - h)$. By lower semicontinuity of $\Phi$, we have
\begin{align*}
\liminf_{n \to \infty} \Phi(x_n,y_n) \geq \Phi(\widehat{x},\widehat{y}) = w(\widehat{y}+h) - \varphi(\widehat{y}+h) \geq 0.
\end{align*}
By \eqref{eq:phileqzero} we have 
\begin{align*}
\lim_{n \to \infty} \Phi(x_n,y_n) = 0 = w(\widehat{y}+h) - \varphi(\widehat{y}+h).
\end{align*}
Since $w - \varphi$ has a strict global minimum at $x_0$, we conclude that $x_0 = \widehat{y} + h$. Since $\widehat{x} = \widehat{y} + 2h$, we have $(\widehat{x}, \widehat{y}) = (x_0 + h, x_0 - h)$. As $x_0 \pm h \in \Omega_\vare$, there exists $N > 0$ such that$(x_n, y_n) \in \Omega_\vare \times \Omega_\vare$ when $n > N$. Let $\psi_1$ and $\psi_2$ be given by
\begin{align*}
 \psi_1(x) = -u(y_n) + 2\varphi\left(\frac{x+y_n}{2}   \right) - \alpha_n\abs{x - y_n - 2h}^2 \\
\psi_2(y) = -u(x_n) + 2\varphi\left(\frac{x_n + y}{2}   \right) - \alpha_n \abs{x_n - y - 2h}^2.
\end{align*}
By construction, $u - \psi_1$ has a local minimum at $x_n$ and $u - \psi_2$ has a local minimum at $y_n$, $D\psi_1(x_n) = p_n - 2q_n$, and $D\psi_2(y_n) = p_n + 2q_n$ where $p_n = D\varphi\left(\frac{x_n +y_n}{2}   \right)$ and $q_n = \alpha_n(x_n - y_n - 2h)$. Since $u$ satisfies $L(Du) = \rho^{1/d}$ on $\Omega_\vare$, we have $L(p_n + 2q_n) \geq \rho(y_n)^{1/d}$ and $L(p_n - 2q_n) \geq \rho(x_n)^{1/d}$ for $n > N$. Using concavity of $L$, we have
\begin{align*}
L(p_n) &= L\left(\frac{p_n + 2q_n}{2} + \frac{p_n - 2q_n}{2} \right) \\
&\geq \frac{1}{2} (L\left(p_n + 2q_n \right) + L\left(p_n - 2q_n \right)) \\
&\geq \frac{1}{2}(\rho(x_n)^{1/d} + \rho(y_n)^{1/d}).
\end{align*}
Using $(\widehat{x}, \widehat{y}) = (x_0 + h, x_0 - h)$, lower semicontinuity of $L$, and semiconvexity of $\rho^{1/d}$, we have
\begin{align*}
L(D\varphi(x_0)) \geq \limsup_{n \to \infty} L\left(D\varphi\left(\frac{x_n + y_n}{2} \right)  \right) &\geq \frac{1}{2}(\rho(x_0+h)^{1/d} + \rho(x_0 - h)^{1/d}) \\
&\geq \rho(x_0)^{1/d} - \frac{K_\rho}{2} \abs{h}^2.
\end{align*}
It follows that $L(Dw) \geq \rho^{1/d} - \frac{K_\rho}{2} \abs{h}^2$ on $\Omega$. Given $\theta > 0$, set $\theta = \frac{1}{2}\abs{h}^2  \max(K_u, \rhomin^{-1/d} K_\rho)$ and let $w_\theta = (1+\theta)w + \theta$. By Proposition \ref{prop:perturb} we have
\begin{align*}
    L(Dw_\theta) \geq \rho^{1/d} - \frac{K_\rho}{2} \abs{h}^2 + d \rhomin^{1/d} \theta \geq \rho^{1/d} \text{ on } \Omega.
\end{align*}
Furthermore, by \eqref{eq:semiconvexU} we have $w \geq u - \frac{K_u}{2}\abs{h}^2$ on $\Gamma$. By choice of $\theta$, we have $w_\theta \geq u$ on $\Gamma$. As $\partial \Omega = \Gamma \cup \partial^{\ast}\Omega$ by assumption, we may apply Theorem \ref{thm:maximumprinciple} to conclude $w_\theta \geq u$ on $\ov{\Omega}$. Hence for all $x\in \Omega$ we have
\begin{align*}
w(x) = \frac{u(x+h) + u(x-h)}{2} &\geq u(x) - \theta(1 + \norm{w}_{L^\infty(\Omega)}) \\
&\geq u(x) - \frac{1}{2}\norm{u}_{L^\infty(\Omega_\vare)} \abs{h}^2 \max(K_u, \rhomin^{-1/d} K_\rho).
\end{align*}
and the result follows.
\end{proof}
Next we establish the existence of an approximate solution to (\ref{eq:PDEsemi}) for $R=1$ in a neighborhood of the boundary when $\rho(x)^{1/d} = a + p\cdot (x-x_0)$. The approximate solution is constructed as $w + v$ where $w$ is the solution to (\ref{eq:PDEsemi}) when $\rho=1$ and $v$ solves a related PDE. Given $x_0\in \partial_{1,M}$, $p \in \R^d$, and $a > 0$, let 
\begin{align}\label{eq:definitionofv}
    v(x) = \frac{1}{2} a^{-\frac{d-1}{d}} ((p\cdot x)((x_1\ldots x_d)^{1/d} - (x_1\ldots x_d)^{-1/d}) - 2(p\cdot x_0)((x_1\ldots x_d)^{1/d}-1)  )
\end{align} 
and
\begin{align}\label{eq:defintiionofw}
    w(x) = a^{1/d}d(x_1\ldots x_d)^{1/d} - a^{1/d}d
\end{align}
and
\begin{align}\label{eq:defintiionofubar}
    \ov{u} = w + v.
\end{align}
\begin{theorem}\label{thm:asymptotic}
Given $x_0 \in \partial_{1,M}\Omega$, $p\in \R^d$, and $a > 0$, let $v$ and $w$ be as in \eqref{eq:definitionofv} and \eqref{eq:defintiionofw}.
\begin{enumerate}
    \item[(a)]
    We have
\begin{equation}\label{eq:PDEforv}
\left\{\begin{aligned}
 \sum_{i=1}^d \left(\prod_{j\neq i} w_{x_j}  \right) v_{x_i}  &=  p\cdot (x-x_0)&&\text{in }\Omega_{1,M}\\ 
v &= 0 &&\text{on } \partial_{1,M} \Omega.
\end{aligned}\right.
\end{equation}
    \item[(b)]
     Let $\ov{u}= w + v$. Then for all $\vare \leq C_d (a M^{d} (1+\abs{p}))^{-1}$ we have
\begin{equation}\label{eq:PDEforubar}
\left\{\begin{aligned}
\ov{u}_{x_1} \ldots \ov{u}_{x_d}  &= a + p\cdot(x-x_0) + E(x)&&\text{in } B(x_0, \vare)\\ 
\ov{u} &= 0 &&\text{on } B(x_0, \vare)\cap \partial_{1,M} \Omega
\end{aligned}\right.
\end{equation}
where $\abs{E} \leq C_d a^{-1} \abs{p}^2 M^2  \vare^2$. Furthermore, $\ov{u}$ is nondecreasing in each coordinate within $B(x_0,\vare) \cap \Omega_{1,M}$.
\end{enumerate}
\end{theorem}
\begin{proof}
We first prove (a). Using \eqref{eq:definitionofv}, it is clear that $v = 0$ on $\partial_{1,M}\Omega$. Furthermore, we have
\begin{align*}
     2a^{\frac{d-1}{d}} x_i v_{x_i} &= \left(x_i p_i + \frac{p\cdot (x-2x_0)}{d}   \right)(x_1\ldots x_d)^{1/d} + \left(-x_i p_i + \frac{p\cdot x}{d}   \right)(x_1\ldots x_d)^{-1/d}
\end{align*}
and
\begin{align*}
         2a^{\frac{d-1}{d}}\sum_{i=1}^d x_i v_{x_i} = 2p\cdot (x-x_0)(x_1\ldots x_d)^{1/d}.
\end{align*}
and it follows that \eqref{eq:PDEforv} is satisfied. To see that \eqref{eq:PDEforubar} holds, observe that
\begin{align*}
    \prod_{i=1}^d (w_{x_i} + v_{x_i}) = w_{x_1}\ldots w_{x_d} + \sum_{i=1}^{d} v_{x_i} \left( \prod_{j\neq i} w_{x_j} \right)  + E = a + p\cdot (x-x_0) + E
\end{align*}
where 
\begin{align*}
     E = \sum_{\ell=2}^d \sum_{\substack{K \subset \set{1,\ldots ,d} \\ \abs{K}=\ell}} \prod_{i\in K} v_{x_i} \prod_{j\notin K} w_{x_j}.
\end{align*}
To bound $\abs{E}$ within $B(x_0,\vare)$, we establish some estimates on $v_{x_i}$. Using that $\max_{x\in \Omega_{1,M}} \frac{1}{x_i} = M^{d-1}$, it is straightforward to verify that $Dv(x_0) = 0$ and 
\begin{align}\label{eq:vxixjbound}
    \abs{v_{x_i x_j}(x_0)} \leq \frac{C_d a^{-\frac{d-1}{d}} M^{d}\abs{p}}{x_{0,i}}
\end{align}
for $1\leq j \leq d$. Hence, we have
\begin{align*}
    \abs{v_{x_i}} \leq \frac{C_d a^{-\frac{d-1}{d}} M^d \abs{p} \vare}{x_{0,i}} \text{ in } B(x_0,\vare).
\end{align*}
Since $w_{x_i} = a^{1/d}(x_1\ldots x_d)^{1/d} x_i^{-1}$, for $x\in B(x_0,\vare)$ and $\vare \leq M^{-(d-1)}$ we have 
\begin{align*}
    \abs{w_{x_i}} \leq \frac{ a^{1/d}}{x_{0,i}} + \frac{Ca^{1/d}\vare}{x_{0,i}^2} \leq \frac{ Ca^{1/d}}{x_{0,i}}. 
\end{align*}
Then for any $K \subset \set{1,\ldots ,d}$ with $\abs{K} = \ell$ we have
\begin{align*}
   \prod_{i\in K} \abs{v_{x_i}} \prod_{j\notin K} \abs{w_{x_j}} &\leq C_d (a^{-\frac{d-1}{d}} M^d \abs{p} \vare)^\ell \prod_{i\in K} x_{0,i}^{-1} \prod_{j\notin K} \frac{ C a^{1/d}}{x_{0,j}}\\
   &\leq  C_d a^{1-\ell}  \left( M^d \abs{p} \vare \right)^{\ell}.
\end{align*}
It follows that $\abs{E} \leq C_d a^{-1} \abs{p}^2 M^{2d}  \vare^2$ for $\vare \leq \frac{1}{M^d (1+\abs{p})}$. To verify that $\ov{u}$ is nondecreasing in each coordinate within $B(x_0,\vare) \cap \Omega_{1,M}$, observe that in $B(x_0,\vare)$ we have
\begin{align*}
    \abs{w_{x_i x_j}(x_0)} \leq \frac{C_d a^{1/d} M^{d}}{x_{0,i}}.
\end{align*}
Using \eqref{eq:vxixjbound}, we have
\begin{align*}
    \abs{\ov{u}_{x_i x_j}(x_0)} \leq \frac{C_d a^{-\frac{d-1}{d}} M^{d}\abs{p}}{x_{0,i}}.
\end{align*}
Using $\ov{u}_{x_i}(x_0) = a^{1/d} x_{0,i}^{-1}$, when $\vare \leq C_d a M^{-d} (1+\abs{p})$ and $x\in B(x_0,\vare)$ we have
\begin{align*}
    \ov{u}_{x_i}(x) \geq a^{1/d} x_{0,i}^{-1} - C_d \vare a^{-\frac{d-1}{d}}\abs{p} x_{0,i}^{-1} \geq 0.
\end{align*}
\end{proof}
We can now apply the comparison principle to show that $\ov{u}$ approximates $u$ near the boundary.
\begin{proposition}\label{prop:ubarapproximation}
Let $u$ denote the solution to (\ref{eq:PDEsemi}), where $\rho \in C^2(\ov{\Omega}_{1,M})$ satisfies \eqref{eq:rhobounds}. Given $x_0 \in \partial_{1,M}\Omega$, let $\ov{u} = v + w$ where $v$ and $w$ are as in \eqref{eq:definitionofv} and \eqref{eq:defintiionofw} with $a = \rho(x_0)$ and $p = D\rho(x_0)$. Then there exists a constant $C_{d,\rho,M} > 0$ such that when $\vare \leq C_{d,\rho,M}$ we have
\begin{equation*}
    \sup_{B(x_0,\vare)\cap \Omega_{1,M}}\abs{u - \ov{u}} \leq C_d \vare^3 M^{3d-1} [\rho^{1/d}]_{C^{0,1}(\R^d)} \left(  \rhomin^{-2}\norm{D\rho}_{L^\infty(\partial_{1,M}\Omega)}^2 + \rhomin^{-1} \norm{D^2 \rho}_{L^\infty(\partial_{1,M}\Omega)} \right).
\end{equation*}
\end{proposition}
\begin{proof}
Letting $H(p) = p_1\ldots p_d$, by Theorem \ref{thm:asymptotic} we have
\begin{align*}
    H(D\ov{u})  &= \rho(x_0) + D\rho(x_0)\cdot(x-x_0) + E(x) \text{ in } \Omega_{1,M}
\end{align*}
where $\abs{E} \leq C_d \rhomin^{-1}\norm{D\rho(x_0)}^2 M^{2d} \vare^2$ in $B(x_0,\vare)$. As $\rho \in C^2(\ov{\Omega}_{1,M})$, for $x\in B(x_0,\vare)\cap \Omega_{1,M}$ we have
\begin{align}\label{eq:taylorlowerbound}
    \rho(x) \geq \rho(x_0) + D\rho(x_0)\cdot (x-x_0) - \norm{D^2 \rho(x_0)}_{L^\infty} \vare^2.
\end{align}
Given $\lambda > 0$, by Proposition \ref{prop:perturb} and \eqref{eq:taylorlowerbound} we have in $B(x_0,\vare) \cap \Omega_{1,M}$ that
\begin{align*}
    H((1+\lambda)Du) &\geq (1+\lambda)\rho  \\
    &\geq \rho + d\rhomin \lambda \\
    &\geq \rho(x_0) + D\rho(x_0)\cdot (x-x_0) - \norm{D^2 \rho(x_0)}_{L^\infty} \vare^2 + d\rhomin \lambda.
\end{align*}
Letting $\lambda = C_d \rhomin^{-2} \norm{D\rho}_{L^\infty(\partial_{1,M}\Omega)}^2 M^{2d} \vare^2 + \rhomin^{-1}\norm{D^2\rho}_{L^\infty(\partial_{1,M}\Omega)} \vare^2$, we have
\begin{align*}
    H((1+\lambda)Du) \geq H(D\ov{u}) \text{ in } B(x_0,\vare) \cap \Omega_{1,M}.
\end{align*}
Observe that $(1+\lambda)u = \ov{u} = 0$ on $\partial_{1,M} \Omega$ and $u$ is nondecreasing in each coordinate. We may apply Theorem \ref{thm:maximumprinciple} with $\Omega = \ov{B(x_0,\vare)}\cap \Omega_{1,M}$ and $\Gamma = \ov{B(x_0,\vare)}\cap \partial_{1,M} \Omega$ to conclude that $(1 + \lambda) u \geq \ov{u}$ on $B(x_0,\vare)\cap \Omega_{1,M}$. By Lemma \ref{thm:lipschitzvar} we have $\norm{u}_{L^\infty(B(x_0,\vare) \cap \Omega_{1,M})} \leq C_d M^{d-1} \norm{\rho^{1/d}}_{C^{0,1}(\R^d)}\vare$. Hence, we have
\begin{align*}
   u \geq \ov{u} - \norm{u}_{L^\infty(B(x_0,\vare) \cap \Omega_{1,M})} \lambda \geq \ov{u} - C_d M^{d-1} \norm{\rho^{1/d}}_{C^{0,1}(\R^d)} \lambda.
\end{align*}
From Theorem \ref{thm:asymptotic}, $\ov{u}$ is nondecreasing in each coordinate within $B(x_0,\vare)$. An analogous application of Theorem \ref{thm:maximumprinciple} shows that $(1-\lambda)u \leq \ov{u}$ on $B(x_0,\vare)\cap \Omega_{1,M}$, and we conclude that $\sup_{B(x_0,\vare)\cap \Omega_{1,M}}\abs{u - \ov{u}}\leq C_d M^{d-1} \norm{\rho^{1/d}}_{C^{0,1}(\R^d)} \lambda$.
\end{proof}
Now we establish semiconvexity estimates on $\ov{u}$ in a neighborhood of $\partial_{1,M}$.
\begin{lemma}\label{lem:ubarsemi}
Given $x_0 \in \partial_{1,M}\Omega$, $a > 0$, and $p \in \R^d$, let $\ov{u}$ be as in \eqref{eq:defintiionofubar}. Then the following statements hold.
\begin{enumerate}
    \item[(a)] For all $\eta \in \R^d$ we have
    \begin{align*}
         \eta^{\top} (D^2 \ov{u}(x_0)) \eta \geq - C_d \abs{\eta}^2 (a^{-\frac{d-1}{d}} M^{2d-1} \abs{p} + a^{1/d} M^{2d-2} ).
    \end{align*}
    \item[(b)] There exist values of $a>0$, $p\in \R^d$, $x_0 \in \partial_{1,M} \Omega$ and $\eta \in \R^d$ such that
\begin{equation*}
            \eta^{\top} (D^2 \ov{u}(x_0)) \eta \leq - C_d \abs{\eta}^2 (a^{-\frac{d-1}{d}} M^{2d-1} \abs{p} + a^{1/d} M^{2d-2} ).
\end{equation*}
\end{enumerate}
\end{lemma}
\begin{proof}
(a) We have
\begin{align*}
    \ov{u}_{x_i x_j} &= \frac{1}{2}a^{-\frac{d-1}{d}} \left(\frac{dx_jp_j + dx_i p_i + p\cdot (x-2x_0) + da}{d^2 x_i x_j} - \delta_{ij} \frac{da + p\cdot (x-2x_0)}{dx_i^2}    \right)(x_1\ldots x_d)^{1/d} \\
    &+ \frac{1}{2}a^{-\frac{d-1}{d}}\left(\frac{dx_j p_j + dx_i p_i - p\cdot x}{d^2 x_i x_j} - \delta_{ij} \frac{p\cdot x}{dx_i^2}   \right)(x_1\ldots x_d)^{-1/d}.
\end{align*}
In the following calculation we shall employ the shorthand notation $\eta \cdot x^{-1} := \sum_{i=1}^d \eta_i x_i^{-1}$ and $\abs{x^{-1}}^2 := \sum_{i=1}^d x_{i}^{-2}$. Let $\eta \in \R^d$ with $\abs{\eta} = 1$. Then we have
\begin{align*}
    2a^{\frac{d-1}{d}}\eta^{\top} (D^2 \ov{u}(x_0)) \eta &= \sum_{i,j=1}^d \eta_i \eta_j \left(\frac{2dx_{0,i} p_i + 2dx_{0,j} p_j - 2p\cdot x_0 + da}{d^2 x_{0,i} x_{0,j}} - \delta_{ij} \frac{a}{x_{0,i}^2}   \right) \\
    &= \frac{2}{d} \sum_{i,j=1}^d \eta_i \eta_j \left(\frac{p_i}{x_{0,j}} + \frac{p_j}{x_{0,i}} - \frac{p\cdot x_0}{d x_{0,i} x_{0,j}}     \right) + \frac{a}{d} \sum_{i,j=1}^d \frac{\eta_i \eta_j}{x_i x_j} - a\sum_{i=1}^d \frac{\eta_i^2}{x_{0,i}^2} \\
    &= \frac{2}{d}(\eta \cdot x_0^{-1}) \left( 2(\eta \cdot p) - \frac{1}{d}(p\cdot x_0)(\eta \cdot x_0^{-1}) \right) + \frac{a}{d}(\eta \cdot x_0^{-1})^2 - a\norm{\eta \cdot x_0^{-1}}^2 \\
    &\geq -C_d \left(\abs{p}\abs{x_0^{-1}} + \abs{p}\abs{x_0}\abs{x_0^{-1}}^2 + a\abs{x_0^{-1}}^2 \right)
\end{align*}
Using that $\min_{x\in \ov{\Omega}_{1,M}} x_i = M^{-d+1}$, we have
\begin{align*}
    2a^{\frac{d-1}{d}}\eta^{\top} (D^2 \ov{u}(x_0)) \eta &\geq - C_d \left(M^{2d-1}\abs{p} + aM^{2d-2} \right).
\end{align*}
(b) Let $a=1$, $\eta = e_1$, and define $x_0 \in \partial_{1,M} \Omega$ by $x_{0,1} = \frac{1}{M^{d-1}}$, $x_{0,j} = M$ for $j=2,\ldots ,d$ and $p = -\frac{v}{\norm{v}}$ where $v = 2e_1 - \frac{x_0}{dx_{0,1}}$. Then $\left( M^{2d-1} \abs{p} + aM^{2d-2} \right) \leq 2 M^{2d-1}$. We have
\begin{align*}
    \eta^{\top} (D^2 \ov{u}(x_0)) \eta &= \frac{2}{dx_{0,1}} p\cdot \left(2e_i - \frac{x_0}{dx_{0,1}} \right) - \frac{1-(1/d)}{x_{0,1}^2} \\
    &\leq  -\frac{2}{d}M^{d-1} \norm{v} \\
    &= -\frac{2}{d}M^{d-1} \sqrt{(2-d^{-1})^2 + (d-1)d^{-1}M^{2d}} \\
    &\leq -C_d M^{2d-1}. \qedhere
\end{align*}
\end{proof}
\begin{remark}
Result (a) establishes an upper bound on the semiconvexity constant of $\ov{u}$, while result (b) shows that this is the best bound (up to a constant $C_d$) that can hold without additional restrictions on $\rho$. In (a)  if we assume also that $p\cdot x_0 \leq 0$ the result can be improved to 
\begin{align*}
        \eta^{\top} (D^2 \ov{u}(x_0)) \eta \geq - C_d (a^{-\frac{d-1}{d}}\abs{p} M^{d-1} + a^{1/d} M^{2d-2}).
\end{align*}
\end{remark}
\begin{proof}[Proof of Theorem \ref{thm:mainsemiconvexity}]
We will prove the result in two steps, first considering the $R=1$ case, and then proving the general case using a scaling argument. Given $M \geq 1$, let $u$ denote the solution of 
\begin{equation}\label{eq:PDEsemiR1}
\left\{\begin{aligned}
(u_{x_1}u_{x_2}\ldots u_{x_d})^{1/d}  &=  \rho^{1/d}&&\text{in }\Omega_{1,2M}\\
u &= 0 &&\text{on } \partial_{1,2M} \Omega.
\end{aligned}\right.
\end{equation}
Letting $w(x) = d\rhomax^{1/d} (x_1\ldots x_d)^{1/d} - d\rhomax^{1/d}$, we have $w = 0$ on $\partial_{1,2M}$ and $(w_{x_1}\ldots w_{x_d})^{1/d} = \rhomax^{1/d} \geq u$ on $\Omega_{1,2M}$. By Theorem \ref{thm:maximumprinciple} we have $u \leq w$ on $\Omega_{1,2M}$, hence $\norm{u}_{\Omega_{1,2M}} \leq C_d M \rhomax^{1/d}$. Given $\vare>0$, let $h\in \R^d$ with $\abs{h} = \vare$ and set $\Gamma_\vare = \set{x\in \Omega_{1,M}: \dist(x,\partial_{1,M} \Omega) < \vare}$ and $U_\vare = \Gamma_{3\vare} \setminus \Gamma_\vare$. Given $x\in U_\vare$, there exists $x_0 \in \partial_{1,M} \Omega$ such that $x\in B(x_0,3\vare)$. Let $\ov{u}$ be as in Proposition \ref{prop:ubarapproximation}, with $a = \rho(x_0)$ and $p = D\rho(x_0)$. By Proposition \ref{prop:ubarapproximation}, there exists a constant $C_{d,\rho,M} > 0$ such that when $\vare \leq C_{d,\rho, M}$ we have
\begin{align*}
    \sup_{B(x_0,3\vare)\cap \Omega_{1,M}}\abs{u - \ov{u}} \leq C_{d,\rho,M} \vare^3.
\end{align*}
By Lemma \ref{lem:ubarsemi} (a) we have
\begin{align*}
    \eta^{\top} D^2 \ov{u}(x_0)\eta \geq -C_d K_{\ov{u}}
\end{align*}
where 
\begin{align*}
    K_{\ov{u}} = \rhomin^{-(d-1)/d}\norm{D\rho}_{L^\infty(\partial_{1,M}\Omega)} M^{2d-1} + \rhomax^{1/d} M^{2d-2}.
\end{align*}
As $\ov{u}$ is smooth, there exists a constant $C_{d,\rho,M} > 0$ such that
\begin{align*}
    \inf_{y\in B(x_0,3\vare)}\eta^{\top} D^2 \ov{u}(y)\eta \geq -C_d K_{\ov{u}} - C_{d,\rho,M} \vare.
\end{align*}
and it follows that
\begin{align*}
    \ov{u}(x+h) - 2\ov{u}(x) + \ov{u}(x-h) \geq -\abs{h}^2 (C_d K_{\ov{u}} + C_{d,\rho,M} \vare).
\end{align*}
Hence, we have
\begin{align*}
    \frac{u(x+h) - 2u(x) + u(x-h)}{\abs{h}^2} &\geq \frac{\ov{u}(x+h) - 2\ov{u}(x) + \ov{u}(x-h)}{\abs{h}^2}- 4\vare^{-2}\sup_{B(x_0,3\vare)\cap \Omega_{1,M}}\abs{u - \ov{u}} \\ &\geq - (C_d K_{\ov{u}} + C_{d,\rho,M} \vare)
\end{align*}
This holds for all $x\in U_\vare$, hence also for all $x\in \ov{U_\vare}$. Letting $\Omega = \Omega_{1,M}\setminus \ov{\Gamma_\vare}$ and $\Gamma = \set{x\in \Omega_{1,M} : \dist(x,\partial_{1,M}\Omega) = \vare}$, we have $\partial \Omega = \Gamma \cup \partial^\ast \Omega$ and $\set{y \in \R^d : \dist(y,\Omega) < \vare} \subset \Omega_{1,2M}$ for $\vare \leq M$. As $\rho^{1/d}$ is semiconvex on $\Omega_{1,M}$ with semiconvexity constant $\norm{D^2 (\rho^{1/d})}_{L^\infty(\Omega_{1,M})}$, we may apply Theorem \ref{thm:semiconvexboundsplitdomain} to conclude that for all $x\in \Omega_{1,M} \setminus \ov{\Gamma_\vare}$ we have
\begin{align*}
    \frac{u(x+h) - 2u(x) + u(x-h)}{\abs{h}^2} &\geq -C_d(1 + \norm{u}_{L^\infty(\Omega_{1,2M})}) (K_{u} + C_{d,\rho,M}\vare) \\
    &\geq -C_d (1+M\rhomax^{1/d}) (K_{u} + C_{d,\rho,M}\vare)
\end{align*}
where 
\begin{align*}
    K_{u} = \max\left( K_{\ov{u}}, \rhomin^{-1/d}\norm{D^2 (\rho^{1/d})}_{L^\infty(\Omega_{1,M})} \right).
\end{align*}
As this holds for all $\vare < C_{d,\rho,M}$ and $h\in \R^d$ with $\abs{h}=\vare$, we conclude that for all $x\in \Omega_{1,M}$ and $h \in \R^d$ with $x\pm h \in \Omega_{1,M}$ we have
\begin{align*}
    \frac{u(x+h) - 2u(x) + u(x-h)}{\abs{h}^2} \geq -C_d (1+M \rhomax^{1/d}) K_{u}.
\end{align*}
Now we prove Theorem \ref{thm:mainsemiconvexity} in full generality. Let $u$ denote the solution of $\eqref{eq:PDEsemi}$ and let $q(x) = u(Rx)$. Then $q$ satisfies
\begin{equation*}
\left\{\begin{aligned}
q_{x_1}q_{x_2}\ldots q_{x_d}  &=  g&&\text{in } \Omega_{1, R^{-1}M}\\ 
q &= 0 &&\text{on } \partial_{1,R^{-1}M} \Omega
\end{aligned}\right.
\end{equation*}
where $g = R^d \rho(Rx)$. By our $R = 1$ result, $q$ satisfies
\begin{align*}
    \frac{q(x+h) - 2q(x) + q(x-h)}{\abs{h}^2} &\geq -C_d (1+M \rhomax^{1/d}) \max\left( K, R^2\rhomin^{-1/d}\norm{D^2 (\rho^{1/d})}_{L^\infty(\Omega_{R,M})} \right)
\end{align*}
where $K = R^{-2d+3}(E_1+E_2)$ with $E_1 = \rhomin^{-(d-1)/d}\norm{D\rho}_{L^\infty(\partial_{R,M}\Omega)} M^{2d-1}$ and $E_2 = \rhomax^{1/d} M^{2d-2}$. Then there exists a constant $C_\rho > 0$ such that when $R \leq C_\rho$ we have
\begin{equation*}
     \frac{q(x+h) - 2q(x) + q(x-h)}{\abs{h}^2} \geq -C_d(1+M\rhomax^{1/d})R^{-2d+3}\left(E_1 + E_2\right).
\end{equation*}
Hence for all $y\in \Omega_{1, R^{-1} M}$ and $h'\in \R^d$ with $h'\neq 0$, we have
\begin{align*}
    \frac{u(Ry+Rh') - 2u(Ry+Rh') + u(Ry-Rh')}{\abs{Rh'}^2} \geq -C_d(1+M\rhomax^{1/d})R^{-2d+1}\left(E_1 + E_2\right)
\end{align*}

Replacing $Rh'$ with $h$ and $Ry$ with $x$, we conclude that for all $x\in \Omega_{R,M}$ and $h \neq 0$ we have
\begin{equation*}
        \frac{u(x+h) - 2u(x) + u(x-h)}{\abs{h}^2} \geq -C_d(1+M\rhomax^{1/d})R^{-2d+1}\left(E_1 + E_2 \right). \qedhere
\end{equation*}
\end{proof}


\begin{thebibliography}{10}

\bibitem{abbasi2017}
B.~Abbasi, J.~Calder, and A.~M. Oberman.
\newblock Anomaly detection and classification for streaming data using partial
  differential equations.
\newblock {\em SIAM Journal on Applied Mathematics}, In press, 2017.


\bibitem{armstrong2014error}
S.~Armstrong and P.~Cardaliaguet and P.~Souganidis.
\newblock Error estimates and convergence rates for the stochastic homogenization of Hamilton-Jacobi equations
\newblock {\em Journal of the American Mathematical Society}, 27(2):479-540, 2014.

\bibitem{bardi1997}
M.~Bardi and I.~Capuzzo-Dolcetta.
\newblock {\em {Optimal control and viscosity solutions of
  Hamilton-Jacobi-Bellman equations}}.
\newblock Springer, 1997.

\bibitem{bollobas1992height}
B.~Bollob{\'a}s and G.~Brightwell.
\newblock The height of a random partial order: concentration of measure.
\newblock {\em The Annals of Applied Probability}, pages 1009--1018, 1992.

\bibitem{bollobas1988}
B.~Bollob{\'a}s and P.~Winkler.
\newblock {The longest chain among random points in Euclidean space}.
\newblock {\em Proceedings of the American Mathematical Society},
  103(2):347--353, 1988.

\bibitem{calder2018lecture}
J.~Calder.
\newblock Lecture notes on viscosity solutions.
\newblock Online lecture notes available at
  \url{http://www-users.math.umn.edu/~jwcalder/viscosity_solutions.pdf}.

\bibitem{calder2015directed}
J.~Calder.
\newblock Directed last passage percolation with discontinuous weights.
\newblock {\em Journal of Statistical Physics}, 158(45):903--949, 2015.

\bibitem{calder2016direct}
J.~Calder.
\newblock A direct verification argument for the {H}amilton--{J}acobi equation
  continuum limit of nondominated sorting.
\newblock {\em Nonlinear Analysis: Theory, Methods \& Applications},
  141:88--108, 2016.

\bibitem{calder2017numerical}
J.~Calder.
\newblock Numerical schemes and rates of convergence for the
  {H}amilton--{J}acobi equation continuum limit of nondominated sorting.
\newblock {\em Numerische Mathematik}, 137(4):819--856, 2017.

\bibitem{calder2014}
J.~Calder, S.~Esedo\=glu, and A.~O. Hero~III.
\newblock A {H}amilton-{J}acobi equation for the continuum limit of
  non-dominated sorting.
\newblock {\em SIAM Journal on Mathematical Analysis}, 46(1):603--638, 2014.

\bibitem{calder2015PDE}
J.~Calder, S.~Esedo\=glu, and A.~O. Hero~III.
\newblock A {PDE}-based approach to non-dominated sorting.
\newblock {\em SIAM Journal on Numerical Analysis}, 53(1):82--104, 2015.

\bibitem{calder2020convex}
J.~Calder and C.~K. Smart.
\newblock {The limit shape of convex hull peeling}.
\newblock {\em Duke Mathematical Journal}, 169(11):2079--2124, 2020.

\bibitem{chan2001}
T.~Chan and L.~Vese.
\newblock {Active contours without edges}.
\newblock {\em IEEE Transactions on Image Processing}, 10(2):266--277, 2001.

\bibitem{crandall1992}
M.~Crandall, H.~Ishii, and P-L.~Lions.
\newblock {User`s guide to viscosity solutions of second order partial differential equations}.
\newblock {Bulletin of the American mathematical society}, 27(1):1-67, 1992.

\bibitem{deb2002}
K.~Deb, A.~Pratap, S.~Agarwal, and T.~Meyarivan.
\newblock {A fast and elitist multiobjective genetic algorithm: NSGA-II}.
\newblock {\em IEEE Transactions on Evolutionary Computation}, 6(2):182--197,
  2002.

\bibitem{deuschel1995}
J.-D. Deuschel and O.~Zeitouni.
\newblock {Limiting curves for i.i.d.\ records}.
\newblock {\em The Annals of Probability}, 23(2):852--878, 1995.

\bibitem{ehrgott2005}
M.~Ehrgott.
\newblock {\em {Multicriteria Optimization}}.
\newblock Springer, Berlin, Heidelberg, New York, 2005.

\bibitem{fleury2002}
G.~Fleury, A.~O. Hero~III, S.~Yoshida, T.~Carter, C.~Barlow, and A.~Swaroop.
\newblock {Pareto analysis for gene filtering in microarray experiments}.
\newblock In {\em {European Signal Processing Conference (EUSIPCO)}}, 2002.

\bibitem{hammersley1972}
J.~Hammersley.
\newblock {A few seedlings of research}.
\newblock In {\em {Proceedings of the Sixth Berkeley Symposium on Mathematical
  Statistics and Probability}}, volume~1, pages 345--394, 1972.

\bibitem{handl2005}
J.~Handl and J.~Knowles.
\newblock Exploiting the trade-off -- the benefits of multiple objectives in
  data clustering.
\newblock In {\em Evolutionary Multi-Criterion Optimization}, pages 547--560.
  Springer, 2005.

\bibitem{hero2003}
A.~O. Hero~III.
\newblock {Gene selection and ranking with microarray data}.
\newblock In {\em {IEEE International Symposium on Signal Processing and its
  Applications}}, volume~1, pages 457--464, 2003.

\bibitem{hsiao2015}
K.-J. Hsiao, J.~Calder, and A.~O. Hero~III.
\newblock Pareto-depth for multiple-query image retrieval.
\newblock {\em IEEE Transactions on Image Processing}, 24(2):583--594, 2015.

\bibitem{hsiao2012}
K.-J. Hsiao, K.~Xu, J.~Calder, and A.~O. Hero~III.
\newblock {Multi-criteria anomaly detection using Pareto Depth Analysis}.
\newblock In {\em {Advances in Neural Information Processing Systems 25}},
  pages 854--862. 2012.

\bibitem{hsiao2015b}
K.-J. Hsiao, K.~Xu, J.~Calder, and A.~O. Hero~III.
\newblock {Multi-criteria similarity-based anomaly detection using Pareto Depth
  Analysis}.
\newblock {\em {IEEE Transactions on Neural Networks and Learning Systems}},
  2015.
\newblock To appear.

\bibitem{kingman1992poisson}
J.~F.~C. Kingman.
\newblock {\em Poisson processes}, volume~3.
\newblock Oxford university press, 1992.

\bibitem{katzourakis2014introduction}
N.~Katzourakis
\newblock{An introduction to viscosity solutions for fully nonlinear PDE with applications to calculus of variations in L?}
\newblock {\em {Springer}},
  2014.

\bibitem{kumar2010}
A.~Kumar and A.~Vladimirsky.
\newblock {An efficient method for multiobjective optimal control and optimal
  control subject to integral constraints}.
\newblock {\em Journal of Computational Mathematics}, 28(4):517--551, 2010.

\bibitem{logan1977}
B.~F. Logan and L.~A. Shepp.
\newblock {A variational problem for random {Y}oung tableaux}.
\newblock {\em Advances in Mathematics}, 26(2):206--222, 1977.

\bibitem{mitchell2003}
I.~Mitchell and S.~Sastry.
\newblock {Continuous path planning with multiple constraints}.
\newblock In {\em {IEEE Conference on Decision and Control}}, volume~5, pages
  5502--5507, Dec. 2003.

\bibitem{mumford1989}
D.~Mumford and J.~Shah.
\newblock {Optimal approximations by piecewise smooth functions and associated
  variational problems}.
\newblock {\em Communications on Pure and Applied Mathematics}, 42(5):577--685,
  1989.

\bibitem{thawinrak2017high}
W.~Thawinrak and J.~Calder.
\newblock High-order filtered schemes for the {H}amilton-{J}acobi continuum
  limit of nondominated sorting.
\newblock {\em Journal of Mathematics Research}, 10(1), Nov 2017.

\bibitem{ulam1961}
S.~Ulam.
\newblock Monte {C}arlo calculations in problems of mathematical physics.
\newblock {\em Modern Mathematics for the Engineers}, pages 261--281, 1961.

\bibitem{vershik1977}
A.~Vershik and S.~Kerov.
\newblock {Asymptotics of the Plancherel measure of the symmetric group and the
  limiting form of Young tables}.
\newblock {\em Soviet Doklady Mathematics}, 18(527-531):38, 1977.

\end{thebibliography}

\end{document}